\documentclass[11pt,openany]{article}

\usepackage{graphicx}
\usepackage{latexsym}
\usepackage{amssymb}
\usepackage{amsmath}
\usepackage{enumerate}

\usepackage{amsmath}
\usepackage{stmaryrd}
\usepackage{hyperref}
\usepackage{amssymb}
\usepackage{CJK}
\usepackage{color}

\usepackage{amsmath,latexsym,amssymb,amsthm,array,amsfonts}

\usepackage{mathrsfs,dsfont,bbm}
\usepackage{csquotes}
\usepackage{layout}
\usepackage{layout}

\newtheorem{lemma}{Lemma}
\newtheorem{definition}{Definition}
\newtheorem{corollary}{Corollary}
\newtheorem{theorem}{Theorem}
\newtheorem{proposition}{Proposition}
\newtheorem{remark}{Remark}

\def\real{{\mathord{{\rm I\kern-2.8pt R}}}}        
\def\inte{{\mathord{{\rm I\kern-2.8pt N}}}}

\def\sZZ{{\rm Z\kern-2.8ptem{}Z}}

\def\z{{\mathchoice
  {\sZZ}
  {\sZZ}
  {\rm Z\kern-0.30em{}Z}
  {\rm Z\kern-0.25em{}Z} }}
\def\sQQ{{\kern 0.27em \vrule height1.45ex width0.03em depth0em
          \kern-0.30em \rm Q}}
\def\qu{{\mathchoice
    {\sQQ}
    {\sQQ}
  {\kern 0.225em \vrule height1.05ex width0.025em depth0em \kern-0.25em \rm Q}
  {\kern 0.180em \vrule height0.78ex width0.020em depth0em \kern-0.20em \rm Q}
        }}
\def\sCC{{\kern 0.27em \vrule height1.45ex width0.03em depth0em
          \kern-0.30em \rm C}}
\def\complex{{\mathchoice
    {\sCC}
    {\sCC}
  {\kern 0.225em \vrule height1.05ex width0.025em depth0em \kern-0.25em \rm C}
  {\kern 0.180em \vrule height0.78ex width0.020em depth0em \kern-0.20em \rm C}
        }}

\MakeOuterQuote{"}

\newcommand{\ba}{\begin{array}}
\newcommand{\ea}{\end{array}}
\newcommand{\be}{\begin{equation}}
\newcommand{\ee}{\end{equation}}
\newcommand{\bea}{\begin{eqnarray}}
\newcommand{\eea}{\end{eqnarray}}
\newcommand{\beaa}{\begin{eqnarray*}}
\newcommand{\eeaa}{\end{eqnarray*}}

\def\z{\zeta}

\font\tenmath=msbm10 \font\sevenmath=msbm7 \font\fivemath=msbm5
\newfam\mathfam \textfont\mathfam=\tenmath
\scriptfont\mathfam=\sevenmath \scriptscriptfont\mathfam=\fivemath

\def \={{\buildrel {\rm (law)} \over =}}

\def\qed{ \hfill \vrule width.25cm height.25cm depth0cm\smallskip}

\newcommand{\basa}{\begin{assumption}}
\newcommand{\easa}{\end{assumption}}

\newcommand{\bas}{\begin{assum}}
\newcommand{\eas}{\end{assum}}

\newcommand{\ignore}[1]{}
\usepackage{authblk}
\begin{document}

\renewcommand{\thefootnote}{\fnsymbol{footnote}}

\renewcommand{\thefootnote}{\fnsymbol{footnote}}

\title{Sobolev-Wigner spaces}

\author[1]{Charles-Philippe Diez \thanks{charles-philippe.diez@univ-lille.fr}}
\affil[1]{Department of Statistics, The Chinese University of Hong-Kong.}

\renewcommand\Authands{ and }

\maketitle

\begin{abstract}
In this paper, we provide new results about the free Malliavin calculus on the Wigner space first developed in the breakthrough work of Biane and Speicher \cite{BS}. We define in this way the higher-order Malliavin derivatives, and we study their associated {\it Sobolev-Wigner} spaces. Using these definitions, we are able to obtain a free counterpart of the {\it Stroock's formula} and various variances identities. As a consequence, we obtain a sophisticated proof {\it a la \"Ust\"unel, Nourdin and Peccati} (\cite{Us,NP-book}) of the product formula between two multiple Wigner integrals. We also study the {\it commutation relations} (of different significations) on the Wigner space, and we show for example the absence of non-trivial bounded {\it central} Malliavin differentiable functionals and the absence of non-trivial Malliavin differentiable projections.
\end{abstract}

\vskip0.3cm

{\bf 2010 AMS Classification Numbers:}   46L54, 60H07, 60H30.

\vskip0.3cm


\section{Introduction}
Free probability, and in particular the concept of freeness which is modelled on free products instead of tensors products of algebras, and which can be seen as a free analog of the classical (tensor) independence was invented by Dan Virgil Voiculescu during the last 80s to have a deeper understanding of $\Pi_1$ factors and especially the free groups factors. Voiculescu in his breakthrough work has shown its powerful applications in the study of von Neumman algebras and the connection with random matrix theory. Indeed, it has allowed several authors to prove numerous results for von Neumann algebras, especially for the free groups factors $L(\mathbb{F}_n),1\leq n\leq \infty$ such as the absence of the property Gamma, absence of Cartan subalgebras or primeness \cite{Voic1,Voic2,Ge} which were unknown until the appearance of this theory. Voiculescu also discovered an important connection with Gaussian random matrices: the large $N\times N$ limit of Gaussian matrices behave as a semicircular system, which has motivated of a lot of parallelism between free probability and random matrix theory.
Then, two decades ago, in the breakthrough paper of Biane and Speicher \cite{BS}, the first results about the {\it free stochastic calculus of variations} on Semicircular spaces appeared, which is the free counterpart of the celebrated {\it stochastic calculus of variations} onto Gaussian spaces (first developed onto the classical Wigner space) and invented by Malliavin during the $70$'s. Indeed, it was defined in their paper a free analogue of the usual (commutative or bosonic) Brownian motion, which is called the {\it free Brownian motion}. Many results of free stochastic analysis were proved such as Ito integration for biprocesses, functional Ito formula for a subspace of smooth  {\it operator-valued Lipschitz} functions, a $L^2$-decomposition of the Wigner Space, which is exactly the noncommutative analogue of the Wiener chaotic decomposition (identified with bosonic Fock spaces), as well as a free counterpart of the classical Malliavin operators. Several important results on the Wiener space were proved to also hold true in the free context: e.g. a free Clark-Ocone-Bismut formula, the {\it hypercontractivity} of the free Ornstein-Uhlenbeck semigroup by Biane in \cite{Biane}, a free Skorohod integration, and also recently several regularity results about of the analytic distribution of Wigner functionals (c.f Mai \cite{Mai}). By considering analogies between the classical and free case, a lot of progress had been made concerning the "{\it fourth moment theorems}" on the Wigner space were proved (see e.g. the main contribution of Kemp et al. \cite{KNPS}, Nourdin and Peccati \cite{np-p} (free Poisson approximation), Bourguin and Campese \cite{BC}, Cebron \cite{C}, or more recently the author in \cite{Di} by means of free Malliavin calculus (the reader interested in the free Stein's method can consult the constantly updated webpage maintained by Nourdin \url{https://sites.google.com/site/malliavinstein/home} for further results ($q$-Gaussian approximation: Deya, Norredine, Nourdin in \cite{q-gauss}, tetilla law: Deya, Nourdin in \cite{Deya}, multidimensional free Poisson Bourguin in \cite{bourg}, invariances principles for homogeneous sums: Deya, Nourdin in \cite{inv}...).
\bigbreak
Our purpose here is to refine several results about the free Malliavin calculus on the Wigner space. We first recall some details which can be found in the main contribution to this topic by Biane and Speicher \cite{BS}. We then introduce the main operators of this infinite dimensional differential calculus, and we prove several results about them: e.g free Ornstein-Uhlenbeck operator as the {\it directional derivative} with respect to a scale parameter. In a second time, we study the {\it commutation relation} (with different meanings) such as: the commutation relation between the free Malliavin derivative and the free Ornstein-Uhlenbeck semigroup, as well as commutation relations between Malliavin derivative and conditional expectation, and a kind of infinite dimensional dual system in the sense of Voiculescu: the commutator between a smooth Wigner functionals and the {\it left annihilation operator} (which might be of independent interest, as it could help, combined with findings of Charlesworth and Shlyakhtenko in \cite{CS} and Mai, Speicher and Weber in \cite{MSW}, to prove the free analog of the Shigekawa's result \cite{shige} about the absolute continuity of the analytic distribution of multiple Wigner integrals). In the last part of this section we prove the absence of central (in the sense of von Neumann algebra) Wigner functionals in $\mathbb{D}^{1,2}$, which gives another proof that the von Neumann algebra generated by a free Brownian motion (known to be isomorphic to $L(\mathbb{F}_{\infty})$, the main idea being to don't use this fact) is a factor as a consequence of the free Poincaré inequality on the Wigner space, by using ideas which have first appeared in the work of Dabrowski in \cite{DAB}, in the finitely generated case. In a third part, we define higher-order free Malliavin derivatives of Wigner functionals. In this way, we then consider their associated {\it Sobolev-Wigner} (or semicircular) spaces, which are thus the free counterpart of the Gaussian ones (usually called {\it Sobolev-Watanabe} spaces). We then provide several results about their chaotic characterization, and as a main consequence we prove a {\it free Stroock formula}. Finally, we are able to give a more sophisticated proof of the multiplication rule on the Wigner space which is more in spirit with the free Malliavin calculus. Indeed, this result can be seen as a free counterpart of a now well known proof of the product formula on the classical Wiener space, which has first appeared in the book of \"Ust\"unel \cite{ust} and in the book of Nourdin and Peccati \cite{NP-book}. In fact, it is a consequence of a {\it Leibniz rule} for the free Malliavin gradient. This kind of Leibniz rule turns out to first explicitly appear (in the finitely generated case) in the work of Voiculescu for the free difference quotients (see the discussion preceding proposition 4.5 in \cite{V}).

\section{Preliminaries}
\begin{flushleft}
In this section, we recall basic definitions about noncommutative $L^p$-spaces.
\newline
Here $\mathcal{M}$ denotes a von Neumann algebra, equipped with a faithful normal state.
\newline
Now by the GNS construction, $\tau$ defines an inner product on $\mathcal{M}$ by setting for all $x,y \in \mathcal{M}$.
\end{flushleft}
\begin{equation}
    \langle x,y\rangle_{\tau}=\tau (y^*x)\nonumber
\end{equation}
\begin{flushleft}
The completion of $\mathcal{M}$ with respect to the induced norm
$\lVert.\rVert_{\tau}$ is denoted $L^2(\mathcal{M},\tau)$. We will omit to denote the state when its clearly defined and denote $\lVert.\rVert_{\tau}$  as $\lVert.\rVert_{2}$ and $L^2(\mathcal{M},\tau)$ as $L^2(\mathcal{M})$. We can also define in the same way the spaces $L^p(\mathcal{M},\tau)$ for $1\leq p\leq\infty $ by taking the completion with respect to the norm :
\end{flushleft}
\begin{equation}
    \lVert x\rVert_p=\tau(\lvert x\rvert^p)^{\frac{1}{p}}
\end{equation}
where $\lvert x\rvert=(x^*x)^{\frac{1}{2}}$
and $L^{\infty}(\mathcal{M},\tau):=\mathcal{M}$ equipped with the operator norm $\rVert .\lVert$.
\begin{flushleft}
From the von Neumann tensor product $\mathcal{M}\otimes \mathcal{M}^{op}$ (we denote here to avoid confusion, the algebraic tensor product as $\odot$) equipped with operator norm : $\lVert .\rVert_{\mathcal{M}\otimes \mathcal{M}^{op}}$, and the faithful normal state $\tau\otimes \tau^{op}$.
We can consider the Hilbert space $L^2(\mathcal{M}\otimes \mathcal{M}^{op},\tau\otimes \tau^{op})$. 
which can be identified with $HS(L^2(\mathcal{M}))$ which is the space of Hilbert–Schmidt operators on $L^2(M)$ via the following map:
\end{flushleft}
\begin{equation}\label{HS}
    x\otimes y \mapsto \langle y,.\rangle_2 x,\quad x,y\in \mathcal{M}\nonumber.
\end{equation}

\section{Wigner-Ito chaoses}
In this section, we will describe the fundamental concepts of the Wigner space which is the free analog of the classical Wiener space: the notion of multiple Wigner-Ito integrals, the chaotic decomposition of $L^2$-functionals, or the product rule between two multiple Wigner integrals which is valid as we have an  $L^{\infty}$-norm estimate of these multiple Wigner integrals, and which ensures that such operators are bounded.

We first recall (and also because some operators will be needed further in the paper) how to construct a free Brownian motion via the free Fock space which will turns out to be $*$-unitarily isomorphic to the completion of the von Neumann algebra generated by the free Brownian motion with respect to the associated
$L^2$-norm.
\begin{definition}
Let $\mathcal{H}_{\mathbb{R}}$ being a real separable Hilbert space and $\mathcal{H}_{\mathbb{C}}=\mathcal{H}_{\mathbb{R}}\otimes_{\mathbb{R}} \mathbb{C}$ its complexification.
\newline
The free (or full) Fock space  is defined as the completion of
\begin{equation*}
F(\mathcal{H}_{\mathbb{C}})= \mathbb{C}\Omega\oplus\bigoplus_{i=1}^{\infty}\mathcal{H}_{\mathbb{C}}^{\otimes {n}}
, \end{equation*}
where $\Omega$ is the "{\it vaccum vector}" (of norm $1$), and $^{\odot n}$ is the algebraic tensor product (without completion).
\newline
with respect to the following the inner product (we explicitly omit to denote the underlying Hilbert space  $\mathcal{H}_{\mathbb{C}}$):
\begin{equation}
\langle g_{1} \otimes ...\otimes g_{n},h_{1} \otimes ... \otimes h_{m} \rangle=\delta_{n,m}\langle g_1,h_{1}\rangle...\langle g_n,h_{n}\rangle\nonumber
, \end{equation}
\end{definition}
\begin{definition}
We define the following operator in $\mathcal{B}(F(\mathcal{H}_{\mathbb{C}}))$ :
\begin{enumerate}
    \item 
For all $h\in \mathcal{H}_{\mathbb{C}}$ the {\it left creation} operator   $l(h) \in \mathcal{B}(F(\mathcal{H}_{\mathbb{C}}))$ by :
\begin{equation}
    l(h)(g_1\otimes\ldots\otimes g_n)=h\otimes g_1\otimes \ldots\otimes g_n\nonumber
, \end{equation}
\item
For all $h\in \mathcal{H}{\mathbb{C}}$ the left annihilation ({\it left annihilation}) $l(h) \in \mathcal{B}(F(\mathcal{H}_{\mathbb{C}}))$\begin{eqnarray}
l^*(h)\Omega&=&0\nonumber\\
    l^*(h)(g_1\otimes\ldots\otimes g_n)&=&\langle h,g_1\rangle g_2\otimes\ldots\otimes  g_n\nonumber
, \end{eqnarray}
\end{enumerate}
\end{definition}
\begin{definition}
We define the following operator in $\mathcal{B}(F(\mathcal{H}_{\mathbb{C}}))$ :
\begin{enumerate}
    \item 
For all $h\in \mathcal{H}_{\mathbb{C}}$ the {\it right creation} operator   $r(h) \in \mathcal{B}(F(\mathcal{H}_{\mathbb{C}}))$ by :
\begin{equation}
    r(h)(g_1\otimes\ldots\otimes g_n)= g_1\otimes \ldots\otimes g_n\otimes h\nonumber
, \end{equation}
\item
For all $h\in \mathcal{H}_{\mathbb{C}}$ the right annihilation  $r^*(h) \in \mathcal{B}(F(\mathcal{H}_{\mathbb{C}}))$\begin{eqnarray}
r^*(h)\Omega&=&0\nonumber\\
    r^*(h)(g_1\otimes\ldots\otimes g_n)&=&\langle h,g_n\rangle g_1\otimes \ldots\otimes g_{n-1} \nonumber
, \end{eqnarray}
\end{enumerate}
\end{definition}
\begin{definition}
We define the following operator in $\mathcal{B}(F(\mathcal{H}_{\mathbb{C}}))$ :
For all $h\in \mathcal{H}_{\mathbb{C}}$ the {\it semicircular operator}: 
\begin{equation}
S(h):=l(h)+l^*(h) \in \mathcal{B}(F(\mathcal{H}_{\mathbb{C}}))\nonumber
\end{equation}
\end{definition}
\begin{definition}
We also let 
\begin{eqnarray}
\mathcal{S}_{alg}(\mathcal{H}_{\mathbb{R}})=*-alg\left\{S(h),h\in\mathcal{H}_{\mathbb{R}}\right\}\nonumber
, \end{eqnarray}
being the $*$-unital algebra generated by the {\it real} semicircular elements
\end{definition}

\begin{theorem}(Voiculescu \cite{V})
The von Neumann algebra generated by the real field operators (semicirculars) is isomorphic to the free group factor (with numbers of generators depending on the dimension of $\mathcal{H}_{\mathbb{R}}$).
\begin{equation}
    \mathcal{SC}(\mathcal{H}_{\mathbb{R}})=\left\{S(h),h\in\mathcal{H}_{\mathbb{R}}\right\}''\simeq L(\mathbb{F}_{dim(\mathcal{H}_{\mathbb{R}})})\nonumber
\end{equation}
\end{theorem}
And we trivially have the following inclusion 
\begin{eqnarray}
\mathcal{S}_{alg}(\mathcal{H}_{\mathbb{R}})\subset \mathcal{SC}(\mathcal{H}_{\mathbb{R}}) \subset \mathcal{B}(F(\mathcal{H}_{\mathbb{C}}))\nonumber
, \end{eqnarray}
Note that $\Omega$ is a cyclic and separating vector on $\mathcal{SC}(\mathcal{H}_{\mathbb{R}})$, i.e $\overline{\mathcal{SC}(\mathcal{H}_{\mathbb{R}})\Omega}=F(\mathcal{H}_{\mathbb{C}})$, and if $X\in \mathcal{SC}(\mathcal{H}_{\mathbb{R}})$ is such that $ X\Omega=0$, then $X=0$.
\bigbreak
We also denote the vaccum state: $\tau(\:.\:)=\langle .\:\Omega,\Omega\rangle_{F(\mathcal{H}_{\mathbb{C}})}$, which is tracial $\tau(XY)=\tau(YX)$ for $X,Y\in \mathcal{SC}(\mathcal{H}_{\mathbb{R}})$, and we recall that $S(h)$ with $ h\in \mathcal{H}_{\mathbb{R}}$ have semicircle distribution with respect to $\tau$.

\begin{proposition}(Second quantization)
Let $T: \mathcal{H}_{\mathbb{R}}\rightarrow \mathcal{K}_{\mathbb{R}}$ a contraction between real separable Hilbert spaces with complexification $T_{\mathbb{C}}$, then the linear map
\begin{equation}
\mathcal{F}(T)(f_1\otimes \ldots\otimes f_n)=(Tf_1)\otimes \ldots\otimes (Tf_n)\nonumber
\end{equation}
extends to a contraction between $F(\mathcal{H}_{\mathbb{C}})$ to $F(\mathcal{K}_{\mathbb{C}})$.
\end{proposition}
\begin{flushleft}
It turns out that the construction of Wigner chaoses could be done efficiently by using an isomorphism between the free Fock space and the Wigner space: such a construction is usually done via the Wick map or multiple  Wigner-Itô integral $I^S(f)$ generally denoted $I(f)$ when the semicircular process $S$ is fixed. We refer to the article \cite{BS} for a complete exposure.
\end{flushleft}

\begin{definition}
We define the finite Wigner chaos of order $n$ denoted $\mathcal{P}_n$ as the (sub) Hilbert space in $L^2(\mathcal{\mathcal{SC}(\mathcal{H}_{\mathbb{R}}}))$ generated by 
\begin{equation}
\left\{S(h_1)\ldots S(h_k),h_1,\ldots,h_k\in \mathcal{H}_{\mathbb{R}}, 1\leq k\leq n\right\}\cup\left\{1\right\}
\end{equation}
We also define the homogeneous Wigner chaos of order $n$ by:
\begin{equation}
    \mathcal{H}_n:=\mathcal{P}_n \ominus \mathcal{P}_{n-1}
\end{equation}
and we obviously have:
\begin{equation}
\mathcal{P}_n=\bigoplus_{k=0}^{n}\mathcal{H}_k,
\end{equation}
where the direct sum is understand in $L^2$.
\end{definition}
In fact, the projection on tensors of length "$n$" denoted $\pi_n$ from 
\newline
$\mathcal{H}_n$ to $L^2(\mathcal{A},\tau)$:
\begin{equation}
    I_n : h_1\otimes...\otimes h_n \mapsto \pi_n(S(h_1)...S(h_n)),
\end{equation}
can be extended to a linear isometry between $\mathcal{H}_{\mathbb{C}}^{\otimes n}$ to $\mathcal{H}_n$.
This allows to introduce the fundamental theorem which asserts that the free Fock space is unitarily isomorphic to the $L^2$-space generated by the isonormal semicircular process.
\begin{theorem}(Biane, Speicher proposition 5.3.2 in \cite{BS})
    We have the following Wigner-Ito chaotic decomposition:
    \begin{equation}
        L^2(\mathcal{SC}(\mathcal{H}_{\mathbb{R}}))=\bigoplus_{k=0}^{\infty} \mathcal{H}_k\simeq F(\mathcal{H}_{\mathbb{C}})\nonumber
    \end{equation}
\end{theorem}

\begin{flushleft}
When the real separable Hilbert space $\mathcal{H}_{\mathbb{R}}=L^2_{\mathbb{R}}(\mathbb{R}_+)$, this is the free Brownian motion case which can be defined as a centered semicircular process of covariance function:
\begin{equation}
K(t,s):=\tau(S_tS_s)=t\wedge s,\: \mbox{\:$t,s\geq 0$},\nonumber
\end{equation}

We also denote the corresponding Brownian filtration (non-decreasing sequences of von Neumann subalgebras whose union weakly generates $\mathcal{SC}({L^2_{\mathbb{R}}(\mathbb{R}_+)}):=\mathcal{SC}$)
\begin{equation}
\mathcal{A}_t=\left\{S(\mathds{1}_{[0,s]}), 0\leq s\leq t\right\}^{''}\nonumber
\end{equation}
and we recall that the free Brownian motion is a noncommutative martingale with respect to this filtration: $\tau(S_t|\mathcal{A}_u)=S_u$, for $0\leq u\leq t$.
\end{flushleft}

\bigbreak
\textbf{In the sequel, we only focus on free Stochastic analysis on the classical Wigner space and will denote both at convenience $\mathcal{SC}$ or $W^*(\left\{S_t,t\geq 0\right\})$  the von Neumann algebra generated by a free Brownian motion.}

\begin{definition}
We define the Tchebychev polynomials of second kind (in the formal variable $X$) as the sequence: $U_0=1$, $U_1=X$ and recursively:
\begin{equation}
XU_k=U_{k+1}+U_{k-1}
\end{equation}
\end{definition}
Note that these are exactly the orthogonal polynomials associated to the standard semicircular distribution and provide the Wigner Ito-isometry since we have:
\begin{proposition}
For $h\in L^2_{\mathbb{R}}(\mathbb{R}_+), \lVert h\rVert_{L^2(\mathbb{R}_+)}=1$, we have:
\begin{equation}
U_p(S(h))=I_p(h^{\otimes p})\nonumber
\end{equation}
and for $n,m\in \mathbb{N}$, and $h,h'\in L^2_{\mathbb{R}}(\mathbb{R}_+)$ of norm $1$, we have:
\begin{eqnarray}
\langle U_n(S(h)),U_m(S(h')\rangle_{L^2(\mathcal{SC})} = \left\{
    \begin{array}{ll}
        \langle h,h'\rangle_{L^2(\mathbb{R}_+)}^n& \mbox{if } n=m\\
        0 & \mbox{if } n\neq m
    \end{array}
\right.
\end{eqnarray}
\end{proposition}

In this setting, the {\it homegenous Wigner chaos} of order $n\geq 1$ can also be defined as the set of {\it multiple free stochastic integrals} of order $n$.
\begin{definition}
    Let the collection of diagonals, i.e. $D^n\subset \mathbb{R}^n_+$,
    \begin{equation}
        D^n=\left\{(t_1,\ldots,t_n)\in\mathbb{R}^n_+,1\leq i,j\leq n, \exists i\neq j , t_i=t_j \right\}
    , \end{equation}
    
    For a characteristic function
        $f=\mathds{1}_A$, where $A=[u_1,v_1]\times\ldots\times [u_n,v_n]$ with $A\cap D^n=\emptyset$

We define the multiple Wigner-Ito stochastic integral as :
\begin{equation}
    I_n^S(f):=(S_{v_1}-S_{u_1})\ldots(S_{v_n}-S_{u_n})
, \end{equation}
    \end{definition}
Which is then extended linearly for:
\begin{equation}
    f=\sum_{i=1}^k a_i\mathds{1}_{A_i}
, \end{equation}
where each $A_i=[u_1^i,v_i^1]\times\ldots\times[u_n^i,v_n^i]$ are disjoints rectangles with moreover $A_i\cap D^n=\emptyset$.

    \begin{lemma} (Wigner-Ito isometry)
    For $f\in L^2(\mathbb{R}_+^n),g\in L^2(\mathbb{R}_+^m)$ simple functions, we have:
    \begin{equation}
    \tau(I_n(f)^*I_m(g))=\delta_{n,m}\langle g,f\rangle_{L^2(\mathbb{R}_+^n)}
        , \end{equation}
    where $\delta$ denotes the delta Kroenecker symbol.
    \end{lemma}
    By density in $L^2$ of simple functions in $L^2(\mathbb{R}_+)$ and Wigner-Ito isometry, we can extend the map:

    \begin{equation}
        f\mapsto I_n(f)
    , \end{equation}
    for all $f\in L^2(\mathbb{R}_+^n)$.
    which will be generally denoted as in the classical case by the formal symbol
    \begin{equation}
    I_n(f)=\int_{\mathbb{R}_+^n} f(t_1,\ldots,t_n)dS_{t_1}\ldots dS_{t_n}
    \end{equation}
Thus we have that
\begin{equation}
\mathcal{H}_n=\left\{I_n(f), f\in L^2(\mathbb{R}_+^n)\right\}
\end{equation}
    \begin{flushleft}
An important lemma is the following, called the {\it Haagerup-Biane-Speicher} inequality, or more precisely its semicircular version (Haagerup lemma 1.4 in \cite{Haag} for length functions over words of length $n$, $f: F_{\infty}\rightarrow C^*_{red}(F_{\infty})$ and prove in the semicircular case by Biane and Speicher), which implies that our multiple integrals constructed to be in  $L^2(\mathcal{SC},\tau)$ are in fact $\mathcal{SC}$ and more precisely they belong to the $C^*$-algebra generated by $\left\{S_t\right\}_{t\geq 0}$ (theorem 5.3.4 of \cite{BS}). This inequality shows that one surprisingly has a control of the operator norm of a multiple Wigner integral by its $L^2$ norm (contrary to the classical case, where Wiener functionals with finite chaotic expansion have moments of any positive order and are not generally bounded).
\end{flushleft}
\begin{theorem}(Haagerup, Biane Speicher \cite{BS})\label{33}.
Let $n\geq 0$, then for all $F \in \mathcal{H}_n$, we have :
\begin{equation}
    \lVert F \rVert_{L^{\infty}(\mathcal{SC})}\leq (n+1)\lVert F \rVert_{L^2(\mathcal{SC})},
\end{equation}
\end{theorem}
The following proposition ensures that the multiple Wiener-Itô integral behaves well with respect to the product. Indeed, elements in some homogeneous Wigner chaos are bounded operators by the previous theorem \ref{33} and so, we are allowed to multiply them. In fact, this formula linearises the product of two multiple Wigner integrals (and is much simpler than in the classical case, as the combinatoric in free probability is much lighter since it involves a smaller class of partitions: the non-crossing ones), see Nourdin and Peccati section 2.7.3 in \cite{NP-book} for comparison with the classical case. 
\begin{flushleft}
Before stating the result, we begin with the definition of {\it nested} contractions.
\end{flushleft}
\begin{definition}
Let $f \in L^2(\mathbb{R}^n_+)$ and $g \in L^2(\mathbb{R}^m_+)$, for every
\newline
$0\leq p\leq n\wedge m$, we define the {\it nested} contraction of order $p$ of $f$ and $g$ as the element of $L^2(\mathbb{R}^{n+m-2p}_+)$ by:
\begin{equation}
    f\stackrel{p}{\frown} g(t_1,...,t_{n+m-2p})=\int_{\mathbb{R}_+^{p}}f(t_1,...,t_{n-p},s_p,...,s_1)g(s_1,...,s_p,t_{n-p+1}...,t_{n+m-2p})d{s_1}...d{s_p}
\end{equation}
\end{definition}
\begin{proposition}(Biane, Speicher proposition 5.3.3 in \cite{BS}) \label{pprodui}
For all $f \in L^2(\mathbb{R}^n_+)$ and $g \in L^2(\mathbb{R}^m_+)$, we have:
\begin{equation}
    I_n(f)I_m(g)=\sum_{p=0}^{n\wedge m}I_{n+m-2p}(f\stackrel{p}{\frown} g),
\end{equation}
\end{proposition}
\begin{flushleft}
In particular for all $n,m\geq 0$, we have the following {\it Wigner-Ito} isometry: \begin{equation}
\tau(I_n(f)^*I_m(g))=\delta_{n,m}\langle g,f\rangle_{L^2(\mathbb{R}_+^n)}
\end{equation}
\end{flushleft}
\begin{remark}
Given a function $f \in L^2(\mathbb{R}^n_+)$, the adjoint of this function is defined almost everywhere by:
\begin{equation}
    f^*(t_1,...,t_n)=\overline{f(t_n,...,t_1)}\nonumber
\end{equation}
which ensure that $I_n(f)^*=I_n(f^*)$. We then easily deduce that $I_n(f)$ is self-adjoint if and if only $f=f^*$. Such functions are usually called {\it mirror-symmetric} (see \cite{Di}).
\end{remark}

\section{The free Malliavin calculus on the Wigner space}\label{part2}

The classical Malliavin operators (onto gaussian spaces) have a free counterpart in the context of free probability (and only for now onto semicircular spaces) thanks to the breakthrough work of Biane and Speicher in \cite {BS}. This construction could be done efficiently on the Free Fock space and then transferred onto the algebra of field operators by the identification $X\mapsto X\Omega$ where $\Omega$ denotes the "vaccum" vector and the $*$-unital algebra generated by the field operators. Since, we are in presence of a closable operator, we will consider as usual the closure of the gradient (still denoted in the by same symbol), defined on a domain which will be the completion of this unital algebra with respect to the $L^p, p\geq 1$-norms. For sake of clarity, we will also assume standard identifications of spaces as usual in the Malliavin calculus. We will however as mentioned before, and due to the heavy notations, only focus our study on the classical Wigner space, i.e. when $\mathcal{H}_{\mathbb{R}}=L^2_{\mathbb{R}}(\mathbb{R}_+)$ (which will be typically suppressed as the Hilbert space is fixed all along the paper).
\begin{definition}
Fix $p\geq 1$, then the free Malliavin derivative is the unique unbounded closable operator (valued into the $L^p$-integrable biprocesses $\mathcal{B}_p$):
\begin{eqnarray}
    \nabla &: &L^p(\mathcal{\mathcal{SC}},\tau) \rightarrow L^p(\mathbb{R}_+,L^p(\mathcal{\mathcal{SC}},\tau)\otimes L^p(\mathcal{\mathcal{SC}},\tau))\nonumber\\
    &&A\mapsto \nabla A=(\nabla_t A)_{t\geq 0}
\end{eqnarray}
such that for all $h\in L^2_{\mathbb{R}}(\mathbb{R}_+)$, $\nabla(S(h))=h.1\otimes 1$, and that, for all $A,B \in S_{alg}$ (where $S_{alg}$ is the unital $*$-algebra generated by $\left\{S(h),h\in L^2_{\mathbb{R}}(\mathbb{R}_+)\right\}$, we have the derivation property $\nabla(AB)=A.\nabla B+ \nabla A.B$ where the left and right actions are given by the multiplication on the left leg and opposite multiplication on the right leg.
\end{definition}
\begin{definition}
We denote for any $p\geq 1$, $\mathbb{D}^{1,p}$ the domain of $\nabla$ viewed as a closable unbounded operator from $L^p(\mathcal{SC})$ to $\mathcal{B}_p$, and which is defined as the completion of $\mathcal{SC}_{alg}$, with respect to the following norm:
\begin{equation}
    \lVert Y\rVert_{1,p}^p=\lVert Y\rVert_p^p+\lVert \nabla Y\rVert_{\mathcal{B}_p}^p.
\end{equation}
\end{definition}
\begin{definition}
For $F\in \mathcal{SC}_{alg}$ and $h\in L^2(\mathbb{R}_+)$, we define the pairing (directional Malliavin derivative):
\begin{equation}
\nabla^h F:=\langle \nabla F,h\rangle_{L^2(\mathbb{R}_+)}:=\int_{\mathbb{R}_+} \nabla_t F \:\overline{h(t)}dt
\end{equation}
\end{definition}
We know recall the (probably) most important formula, which is at the basis of the free Malliavin calculus as it will provide the {\it integration-by-parts} formula.
\begin{lemma}(NC Wick formula)
Let $F=S(e_1)\ldots S(e_n) \in \mathcal{SC}_{alg}$ where $e_1,\ldots,e_n, h\in L^2_{\mathbb{R}}(\mathbb{R}_+)$, then:
\begin{equation}
\tau(FS(h))=\sum_{k=1}^n\langle e_k,h\rangle_{L^2(\mathbb{R}_+)} \tau(S(e_1)\ldots S(e_{k-1}))\tau(S(e_{k+1})\ldots S(e_n))
\end{equation}
\end{lemma}
As a consequence, we have the fundamental integration-by-parts which will be fundamental in the paper.
\begin{proposition}(Integration-by-parts, Biane-Speicher, lemma 5.2.2 in \cite{BS})
For $F\in \mathcal{SC}_{alg}$, and $h\in L^2_{\mathbb{R}}(\mathbb{R}_+)$, we have
\begin{equation}\label{79}
    \tau(FS(h))=\tau\otimes\tau(\nabla^h F)
\end{equation}
\end{proposition}
Since the free Malliavin gradient enjoys a chain rule, we have as an easy consequence that:
\begin{lemma}\label{80}
Suppose that $X,Y,Z\in \mathcal{SC}_{alg}$ and $h\in L^2_{\mathbb{R}}(\mathbb{R}_{+})$, then:
\begin{equation}
\tau\otimes\tau(X.\nabla^hY .Z )=-\tau\otimes\tau(\nabla^hX. YZ )-\tau\otimes\tau(XY.\nabla^h Z)+\tau(XYZS(h))
\end{equation}
\end{lemma}
In the following set of propositions, we state basic results about the action of the free Malliavin derivative on finite Wigner chaos.
\begin{proposition}(proposition 5.3.10 of \cite{BS})
$\mathbb{D}^{1,2}$ contains $\mathcal{P}_n$, and the restriction of $\nabla$ to this space is a bounded linear operator.
\end{proposition}
\begin{flushleft}
We can also explicit the action of $\nabla$ on $\mathcal{P}_n$. First, we note that for any $n,m\geq 0$, the map $I_n\otimes I_m$ from ${L^2(\mathbb{R}^{n}_+)\otimes L^2(\mathbb{R}^{m}_+)}$ to $\mathcal{H}_n\otimes \mathcal{H}_m$.
By the isomorphism between ${L^2(\mathbb{R}^{n+m}_+)}$ and ${L^2(\mathbb{R}^{n}_+)\otimes L^2(\mathbb{R}^{m}_+)}$, we can see the linear extension of the map:
\begin{eqnarray}
    I_n\otimes I_m &:& L^2(\mathbb{R}^{n+m}_+) \rightarrow \mathcal{H}_n\otimes \mathcal{H}_m\nonumber\\
    &&h_1\otimes...\otimes h_{n+m} \mapsto \pi_n(S(h_1)...S(h_n))\otimes\pi_m(S(h_{n+1})...S(h_{n+m}))
\end{eqnarray}
In particular, we will define for an elementary tensor $f=h\otimes g \in  L^2(\mathbb{R}^{n}_+)\otimes L^2(\mathbb{R}^{m}_+)\simeq L^2(\mathbb{R}^{n+m}_+)$, $I_n\otimes I_m(f):=I_n(h)\otimes I_m(g)$ and then continuously extendy for general $f\in L^2(\mathbb{R}^{n+m}_+)$ viewed as an element of $L^2(\mathbb{R}^{n}_+)\otimes L^2(\mathbb{R}^{m}_+)$.
\end{flushleft}
For all $A\otimes B \in \mathcal{P}_n \otimes \mathcal{P}_n$ and $B\otimes C \in \mathcal{P}_n \otimes \mathcal{P}_n$,we denote :
$(A\otimes B)^{*}=A^*\otimes B^*$ and $(A\otimes B)\sharp (C\otimes D)=AC\otimes DB$, and we extend them by linearity for the first map which turns out to provide a map from $(\mathcal{P}_n \otimes \mathcal{P}_n)$ to $\mathcal{P}_n \otimes \mathcal{P}_n$, and by bilinearity and continuity for the second one which gives a map from $(\mathcal{P}_n \otimes \mathcal{P}_n)^2$ to $P_{2n} \otimes P_{2n}$.
\begin{flushleft}

We define for all $f\in L^2(\mathbb{R}^{n}_+)$, the function $f_t^k \in L^2(\mathbb{R}^{n-1}_+)$ by the equality for almost all $t\geq 0$ :
\begin{equation}
    f(t_1,...,t,t_{k+1},...,t_n)=f_t^k(t_1,..,t_{k-1},t_{k+1},...,t_n)\nonumber
\end{equation}
\end{flushleft}
\begin{proposition}\label{pp3}(proposition 5.3.9 of \cite{BS})
The Malliavin derivative maps $\mathcal{P}_n$ into $L^2(\mathbb{R},\mathcal{P}_n\otimes \mathcal{P}_n)$, indeed for $f \in L^2(\mathbb{R}^{n}_+)$, then for almost all $t\geq 0$ :

\begin{equation}
    \nabla_t(I_n(f))=\sum_{k=1}^n I_{k-1}\otimes I_{n-k}(f_t^k),
\end{equation}
\end{proposition}
\begin{definition}
We denote $\delta: dom(\delta)\subset \mathcal{B}_2\rightarrow L^2(\mathcal{SC})$ the adjoint of the free Malliavin derivative which is a densely defined and closable operator called the {\it free Skorohod integral}, the domain of its closure (still denoted $\delta$) will be denoted as $dom(\delta)$ and it's characterized by the duality relation: $u\in dom(\delta)$, if there exists a (unique) element denoted $\delta(u)\in L^2(\mathcal{SC})$ such that for all $F\in \mathcal{SC}_{alg}$:
\begin{equation}
\langle \nabla F,u\rangle_{\mathcal{B}_2}=\langle F,\delta(u)\rangle_{L^2(\mathcal{SC})},
\end{equation}
\end{definition}
We also restate for convenience the infinite dimensional versions of Voiculescu formulas for the free Skorohod integral (proposition 3.9 in \cite{V}) which can be thus thought as the free analog of the formula
\newline
$\delta(Fh)=F.W(h)-\langle DF,h\rangle$ with $F\in \mathbb{D}^{h,2}$ in the classical case (c.f Nualart proposition 1.3.4 \cite{Nual}).
\begin{proposition}\label{ppV}
Let $F\in \mathcal{SC}_{alg}^{\odot 2}$ and $h\in L^2_{\mathbb{R}}(\mathbb{R}_+)$, then $F.h$ belongs to $dom(\delta)$ and:
\begin{eqnarray}
    \delta(F.h)&=&F\sharp S(h)-m_1\circ(id\otimes \tau\otimes id)(\langle \tilde{\nabla}F,h\rangle_{L^2(\mathbb{R}_+)})\nonumber\\
    &=&F\sharp S(h)-m_1\circ (id\otimes \tau\otimes id)(\langle(\nabla \otimes id +id \otimes \nabla)F,h\rangle_{L^2(\mathbb{R}_+)})\nonumber
\end{eqnarray}
where $m_1$ denotes the linear extension of the multiplication from $\mathcal{SC}_{alg}\otimes \mathcal{SC}_{alg}$ into $\mathcal{SC}_{alg}$ given by:
\begin{equation}
    m_1(a\otimes b)=ab
\end{equation}
\end{proposition}
\begin{remark}
This lat proposition is basically the infinite dimensional version on the Wigner space of the well known Voiculescu formulas applied to the free difference quotient in a semicircular system (c.f section 4 in \cite{V}). We will however prefer to denotes the last term in the second way because it is better fitted to the usual conventions used in Voiculescu formulas and similar others works.
\end{remark}
\begin{flushleft}
We now state the chain rule of the free Malliavin derivative:
\end{flushleft}

\begin{proposition}\label{pp5}
For all $F\in \mathcal{P}_d$, and for all $P\in \mathbb{C}[X]$
\begin{equation}
    \nabla_tP(F)=\partial P(F)\sharp \nabla_t(F),\nonumber
\end{equation}
and for the multivariate case, we have for all $(F_1,...,F_n)$ with each $ F_i\in \mathcal{P}_d$ and $P\in \mathbb{C}\langle F_1,\ldots,F_n\rangle $ :
\begin{equation}
    \nabla_t (P(F))=\sum_{k=1}^n \partial_k P(F)\sharp \nabla_t F_k,\nonumber
\end{equation}
\end{proposition}
We first introduce the class of functions (a subspace of {\it operator-Lipschitz functions}) which can be seen as a "{good}" subset of noncommutative functions to deduce some chain rules in free stochastic analysis: a free Ito formula and a chain rule for the free Malliavin derivative (we do not investigate further here this line of topic which have recently known important achievements: see Nikitopoulos\cite{EN}, or Jekel, Li and Shlyakhtenko and \cite{JWS}). Note also that this class of functions is well approximated by polynomials onto compact interval (c.f theorem 3.20 in \cite{KNPS}).
\begin{definition}
We define the following set of smooth $L^2$-functions as:
\begin{equation}
    \mathcal{I}_2(f)=\left\{f:\mathbb{R}\rightarrow \mathbb{C}, f(x)=\int_{\mathbb{R}} e^{ix}\mu(dy),\int_{\mathbb{R}}\lvert y\rvert^2\lvert\mu\rvert(dy)<\infty,\right\},
    \end{equation}
    where $\mu$ is a complex valued measure on $\mathbb{R}$.
    \end{definition}
\begin{proposition}(generalized chain rule, Biane-Speicher)
Let $F\in \mathbb{D}^{1,2}$ and $\phi\in \mathcal{I}_2(f)$, then $\phi(F)\in \mathbb{D}^{1,2}$ and we have for almost all $t\geq 0$:
\begin{equation}
\nabla_t\phi(F)=\partial \phi(F)\sharp \nabla_tF
\end{equation}

\end{proposition}
\begin{theorem}(Biane, Speicher, proposition 5.3.12 in\cite{BS})\label{BS}
Let $F\in \mathbb{D}^{1,2}$, then we have the following Clark-Ocone-Bismut formula:
\begin{equation}
    F=\tau(F)+\delta(\Gamma\circ \nabla F)
\end{equation}
where $\Gamma$ denotes the orthogonal projection from $\mathcal{B}_{2}$ onto the square integrable adapted biprocesses $\mathcal{B}_{2}^a$ and $\delta$ the Skorohod integral (the adjoint of the Malliavin derivative) and the integration is understand as a free stochastic integral with respect to the free Brownian motion of an adapted biprocesses.
\end{theorem}
\begin{flushleft}
As a straightforward corollary, we have the following representation in fixed finite Wigner chaos:
\end{flushleft}
\begin{theorem}\label{proj}
Fix $n\geq 0$, then we have:
\begin{equation}
id_{|\mathcal{P}_n}=P_1+\delta (\Gamma\circ\nabla)_{|{\mathcal{P}_n}}
\end{equation}
where $P_1$ denotes the orthogonal projection onto $\mathbb{C}.1\subset L^2(\mathcal{SC})$.
\end{theorem}
\begin{theorem}Heinseberg commutation relation (Biane, Speicher proposition 5.4.1 in \cite{BS})
\newline
For $U\in dom(\delta)$ with finite chaotic expansion (more generally for $u\in \mathbb{L}^{1,2}$), we have the following commutation relation between Skorokhod integral and Malliavin derivative:
\begin{equation}
\nabla_t(\delta(U))=\nabla_t(\delta(U))+\delta_s(\nabla_t(U_s))
\end{equation}
where the subscript $\delta_s$ indicates that the Skorohod integration acts in the variable $s$ and on biprocess, we set in this case $\nabla_t(A_s\otimes B_s)=\nabla_t A_s\otimes B_s+ A_s\otimes \nabla_t A_s$
\end{theorem}

\begin{definition}
We define the free Ornstein-Uhlenbeck semigroup $(P_t)_{t\geq 0}$ as the second quantization of the contraction $T_t=e^{-t}Id_{L^2_{\mathbb{R}}(\mathbb{R}_+)}$, that is
\newline
$P_t:=\mathcal{F}(e^{-t}Id_{L^2_{\mathbb{R}}(\mathbb{R}_+)})
,t\geq 0$, which is unital, tracial, completely positive and which acts on
$F=\sum_{n=1}^{\infty}I_n(f_n) \in L^2(\mathcal{SC})$ as:
\begin{equation}
P_tF:=\sum_{n=0}^{\infty}e^{-nt}I_n(f_n)
\end{equation}
\end{definition}
\begin{definition}
We denote $L$, the $L^2$-generator of the free Ornstein-Uhlenbeck semigroup defined for $F=\sum_{n=0}^{\infty}I_n(f_n)\in L^2(\mathcal{SC})$ as:
\begin{equation}
LF:=\sum_{n=0}^{\infty}-nI_n(f_n)
\end{equation}
whenever the last series exists in $L^2$, that is that the domain of $L$ (a closed, non-positive self-adjoint operator with dense domain) denoted $dom(L)$ with respect to the $L^2$ norm is given by the following chaotic characterization: 
\begin{equation}
dom(L)=\left\{F=\sum_{n=0}^{\infty}I_n(f_n), \sum_{n=0}^{\infty}n^2\lVert f_n\rVert^2_{L^2(\mathbb{R}_+^n)}<\infty\right\}:=\mathbb{D}^{2,2}\nonumber
\end{equation}

\end{definition}
\begin{definition}
We also denote $C$, the {\it Cauchy Operator}: the negative square-root of the free number operator $C:=-\sqrt{-L}$, defined for $F=\sum_{n=0}^{\infty}I_n(f_n)\in L^2(\mathcal{SC})$ as:
\begin{equation}
CF:=\sum_{n=0}^{\infty}-nI_n(f_n)
\end{equation}
with domain denoted $dom(C)$: 
\begin{equation}
dom(C)=\left\{F=\sum_{n=0}^{\infty}I_n(f_n), \sum_{n=0}^{\infty}n\lVert f_n\rVert^2_{L^2(\mathbb{R}_+^n)}<\infty\right\}:=\mathbb{D}^{1,2}\nonumber
\end{equation}

\end{definition}

In fact, $-L$ is an unbounded positive self-adjoint operator with pure point spectrum: $sp(-L)=\mathbb{N}$ with for each $n\geq 0$ the associated Wigner chaos of order $\left\{\mathcal{H}_n,n\geq 0\right\}$ as corresponding eigenspace, we have the following proposition which is a very important feature of the {\it free group factors} as the "{\it free fourth moment phenomenon}" on the Wigner space precisely appears as in the classical case for eigenvalues of this operator (see Ledoux \cite{ML} for reference in the classical case and Kemp et al. or Cébron or the author \cite{KNPS,C,Di} for more precise details and a conjecture in the free case).
\begin{proposition}
We have the following abstract chaotic decomposition:
\begin{equation}
L^2(\mathcal{\mathcal{SC}})=\bigoplus_{k=0}^{\infty}ker(L+kId)
\end{equation}
\end{proposition}
As in the classical case, the Ornstein-Ulhenbeck operator can be seen as a {\it directional derivative} with respect to a scale parameter (see Nualart and Zakai, section 2 in \cite{NuZ}). This also has motivated the definition of $L$ as the {\it derivative} operator in the early beginning of Malliavin calculus.
\begin{proposition}
Let $F=\sum_{n=1}^{\infty}I_n(f) \in L^2(\mathcal{SC})$, define for $\lvert \lambda\rvert <1$, the functional:
\begin{equation}
F_{\lambda}=\tau(F).1+\sum_{n=1}^{\infty} \lambda^n I_n(f),
\end{equation}
Then $LF$ exists if and if only: $F^{\varepsilon}:=\frac{1}{\varepsilon}(F_{1-\varepsilon}-F)$ converges in $L^2$ as $\varepsilon$ tends to zero and in this case,
\begin{equation}
LF=\underset{\varepsilon\rightarrow 0}{\lim}\frac{1}{\varepsilon}(F_{1-\varepsilon}-F),
\end{equation}
\end{proposition}
\begin{proof}
For $0<\varepsilon<1$, for all $n\geq 0$, $\varepsilon^{-1}((1-\varepsilon)^n-1)\leq n$ (and converges to $-n$ as $\varepsilon\rightarrow 0$) , if $LF$ exists then a simple computation shows that:
\begin{equation}
\lim_{\substack{{\varepsilon_1\rightarrow 0}\\
{\varepsilon_2\rightarrow 0}}} \tau(F^{\varepsilon_1}F^{\varepsilon_2})=\tau[\lvert LF\rvert ^2],\nonumber
\end{equation}
thus the only if part is obtained.
\newline
For the reverse implication (let's recall that $\pi_n$ denotes the projection onto the $n$-th Wigner chaos) suppose that $F^{\varepsilon}$ is a Cauchy-sequence in $L^2$ converging to $G\in L^2(\mathcal{SC})$. 
\newline
Then we set, $F_n=\sum_{k=0}^n \pi_n(F)$,  $F_n^{\varepsilon}=\sum_{k=0}^n \pi_n(F^{\varepsilon})$, and 
$G_n=\sum_{k=0}^n \pi_n(G)$, then we have that (since we suppose that the limit is $G$) $F_n^{\varepsilon}\overunderset{L^2}{\varepsilon \rightarrow 0}{\rightarrow} G_n$ and by assumption that: $F_n^{\varepsilon}\overunderset{L^2}{\varepsilon \rightarrow 0}{\rightarrow} LF_n$, thus $LF$ exists and $F^{\varepsilon}\overunderset{L^2}{\varepsilon\rightarrow 0}{\rightarrow} LF$ and the proof is obtained.
\qed
\end{proof}

\begin{flushleft}
    It turns out that, as in the classical case, the semigroup $(P_t)$ satisfies a commutation relation with the free Malliavin gradient. This result has deeps consequences in the classical case, the main one is with the logarithmic Sobolev inequality. Unfortunately, we have not been able to use this fact prove it in differently of Biane (c.f corollary 1 in \cite{Biane}) by means of free Ito formula and noncommutative martingales. In fact, in the classical case, the most simple and direct proof is achieved by means of Ito formula and the use of martingales and seems to first appeared in the introduction of Capitaine, Hsu and Ledoux in \cite{chl} (this proof seems to be known for a long time by Maurey and Neveu as discussed in the introduction of the aforementioned paper). 
    \end{flushleft}
\begin{remark}
The free Ornstein-Uhlenbeck semigroup acts in fact as a mollifier: for all , 
$P_t(L^2(\mathcal{SC}))\subset \mathbb{D}^{1,2}$ as we have for 
$F=\sum_{n=0}^{\infty}I_n(f_n)\in L^2(\mathcal{SC})$, 
$\sum_{n=0}^{\infty}ne^{-2nt}\lVert f_n\rVert^2_{L^2(\mathbb{R}_+^n)}<\infty$.

\end{remark}
\begin{proposition}(Commutation relation with O.U semigroup)\label{bismut}
Let $F\in \mathbb{D}^{1,2}$, then we have the following relation which holds true for every $t>0$:
\begin{equation}
    \nabla P_tF=e^{-t}P_t^{\otimes 2}\nabla F
\end{equation}
\end{proposition}
\begin{proof}
Thanks to an approximation argument, it suffice to prove that the result holds true for, $F=I_n(f), f\in L^2(\mathbb{R}_+^n)$, $n\geq 0$.
\newline
Then for all $t>0$, and almost all $s\geq 0$,
\begin{eqnarray}
    \nabla_s P_tF&=&\nabla_s (e^{-nt}I_n(f))\nonumber\\
    &=&e^{-nt}\sum_{k=1}^n I_{k-1}\otimes I_{n-k}(f_s^k)\nonumber\\
    &=&e^{-t}\sum_{k=1}^n e^{-(k-1)t}I_{k-1}\otimes e^{-(n-k)t}I_{n-k}(f_s^k)\nonumber\\
    &=& e^{-t}(P_t\otimes P_t)\nabla_sF,\nonumber
\end{eqnarray}
which concludes.
\qed
\end{proof}

\begin{flushleft}
We can also provide some consequences of the {\it hypercontractivity} of the free Ornstein-Uhlenbeck semigroup, this last important result was proved by Biane in the free case by \cite{Biane} (in fact, Biane proved a much stronger result: the hypercontractivity of the $q$-Gaussian Ornstein-Uhlenbeck semigroup for all $q\in (-1,1)$, thus the free case $q=0$ being just a particular result). We also remind tho the reader that the next result which is basically an equivalence in the finite Wigner chaos between $\lVert\:.\rVert_p$ and $\lVert \:.\rVert_q$ is an important lemma in the classical case to reach the proof of the so-called {\it Meyer inequalities}, first proved by Meyer in the celebrated paper \cite{meyer}. Proving the free version of the {\it Meyer's inequalities} is for now a longstanding open problem in the free probabilistic context (we don't even know what norm should use, as it relies on noncommutative Riesz transform and thus to prove equivalence on the $L^p,1\leq p\leq\infty$-norms between a noncommutative {\it Carré-du-champ} and the {\it Cauchy operator}, i.e. $\lVert \Gamma(F,F)\rVert_p\sim c_p\lVert CF\rVert_p$ where $\Gamma(F,F)=\langle \nabla F,\nabla F\rangle_{L^2(\mathbb{R}_+)}=\int_{\mathbb{R}_+}\nabla_t(F)\sharp (\nabla_t(F))^*dt$ is the noncommutative {\it Carré-du-champ} on the Wigner space) as mentioned by Biane and Speicher in the last page of \cite{BS} (see also Nualart \cite{Nual} or \"Ust\"unel \cite {ust} for details about classical {\it Meyer's inequalities} and different types of proofs). 
\end{flushleft}
\begin{theorem}
    Let $F\in L^p(\mathcal{SC}),p\geq 1$ and denote the projection onto the $n$-th Wiener chaos $\pi_n(F)$. Then the linear map:
    \newline
    \begin{equation}
    F\mapsto \pi_n(F)
        \end{equation}
        is continuous on $L^p(\mathcal{SC})$.
\end{theorem}
\begin{proof}
    Let's first suppose that $p\geq 2$, and choose $t\geq 0$, such that $p=e^{2t}+1$. Then, from the hypercontractivity of the free Ornstein-Uhlenbeck semigroup (Theorem $3$ of Biane in \cite{Biane})  we get:
    \begin{equation}
        \lVert P_tF\rVert_{p}\leq \lVert F\rVert_2\leq \lVert F\rVert_p\nonumber
    \end{equation}
Now since $(P_t)$ is a contraction for all $L^p(\mathcal{SC}),p\geq 1$-norms, we know that:
\begin{equation}
    \lVert P_t\pi_n(F)\rVert_{p}\leq \lVert \pi_n(F)\rVert_2\leq \lVert F\rVert_2.\nonumber
\end{equation}
but 
\begin{equation}
    P_t\pi_n(F)=e^{-nt}\pi_n(F)
\end{equation}
and thus 
\begin{equation}
    \lVert \pi_n(F)\rVert_p\leq e^{nt}\lVert F\rVert_p\nonumber
\end{equation}
Now if $1<p<2$, we use the duality where $q$ is the conjugate of $p$ and we choose $t\geq 0$ such that $q=e^{2t}+1$ to get:
\begin{eqnarray}
    \lVert \pi_n(F)\rVert_p&=&\underset{\lVert G\rVert_q\leq 1}{\sup}\lvert\langle G,\pi_n(F)\rangle_2\rvert\nonumber\\
    &=&\underset{\lVert G\rVert_q\leq 1}{\sup}\rvert\langle \pi_n(G),F\rangle_2\rvert\nonumber
\nonumber\\
&=&\underset{\lVert G\rVert_q\leq 1}{\sup}\lvert\langle \pi_n(F),\pi_n(G)\rangle_2\rvert\nonumber\\
&\leq &\lVert F\rVert_p \underset{\lVert G\rVert_q\leq 1}{\sup}\lVert \pi_n(G)\rVert_q\nonumber\\
&\leq & e^{nt}\lVert F\rVert_p,\nonumber\\
\end{eqnarray}
\end{proof}

\section{Commutation relations and Malliavin derivatives}

The goal of this section is to prove several commutation relation (of different kinds) between the Malliavin derivative and conditional expectation, and also between a Malliavin differentiable element and the {\it left creation operator}. We also investigate the absence of central and Malliavin differentiable elements in the Wigner space. We also show that there is no Malliavin differentiable idempotents in the Wigner space.
\newline
 Most of the result of this section can thus be thought as free analog of well-known results which can be found in the celebrated book of Nualart \cite{Nual}.
\begin{flushleft}
We first recall that given a von Neumann subalgebra $\mathcal{B}$ of $W^*(\left\{S_t,t\geq 0\right\})$, there exists a unique trace preserving conditional expectation onto $\mathcal{B}$ denoted $\tau[.|\mathcal{B}]$. 
\newline
The interested reader can consult the paper of Takesaki \cite{take} to have further details on the notion of conditional expectation onto von Neumann subalgebras.
\end{flushleft}
As usual, for any Borel set $\mathcal{A}\subset (\mathbb{R}_+,\mathcal{B}(\mathbb{R}_+))$ with finite Lebesgue measure $\lambda(\mathcal{A})$, we denote $S(\mathcal{A}):=S(\mathds{1}_{\mathcal{A}})$, the associated semicircular element (centered and of variance $\lambda(\mathcal{A})$).
\begin{definition}
Given any Borel set $\mathcal{A}$ of $(\mathbb{R}_+,\mathcal{B}(\mathbb{R}_+))$ with finite Lebesgue measure, we denote $\mathcal{F}_{\mathcal{A}}$ the von Neumann subalgebra generated by:
\begin{equation}
    \left\{S(\mathds{1}_{\mathcal{B}}), B\subset \mathcal{A}, \lambda(B)<\infty\right\}
\end{equation}
\end{definition}

\begin{lemma}
Let $F=\sum_{n=0}^{\infty}I_n(f_n)\in L^2(\mathcal{SC})$, then:
\begin{equation}
    \tau(F|\mathcal{F}_{\mathcal{A}})=\sum_{n=0}^{\infty} I_n(f_n\mathds{1}_{\mathcal{A}})
\end{equation}
\end{lemma}
\begin{proof}
Let consider an elementary multiple Wigner integral, $F=I_n(f_n)$ with $f_n=\mathds{1}_{B_1\times\ldots\times B_n}$, where $(B_i)_{i=1}^n$ are mutually disjoints Borel sets with finite Lebesgue measure (and note also that the mapping $B\mapsto S(B)$ is linear),
\newline
Then by using freeness, we see that:
\begin{eqnarray}
\tau(F|\mathcal{F}_{\mathcal{A}})&=&\tau(S(B_1)\ldots S(B_n)|\mathcal{F}_{\mathcal{A}})\nonumber\\
&=&\tau\left(\prod_{i=1}^n(S(B_i\cap\mathcal{A})+S(B_i\cap\mathcal{A}^c))|\mathcal{F}_{\mathcal{A}}\right)\nonumber\\
&=& S(B_1\cap\mathcal{A})\ldots S(B_n\cap\mathcal{A})\nonumber\\
&:=& I_n(f_n\mathds{1}_{\mathcal{A}}^{\otimes n})\nonumber
\end{eqnarray}
Where we have used in the third line that $B_i\cap B_j=\emptyset$ for $ i\neq j$. The general result follows then by approximation, and polarization.
\qed
\end{proof}
\begin{flushleft}
We first introduce an elementary lemma which ensures that conditional expection onto von Neumann subalgebra (of the one generated by the free Brownian motion) of Malliavin differentiable functionals is also Malliavin differentiable also provide the commutation relation between the derivative and the conditional expectation. The reader familiar with Malliavin calculus might notice that this fact is particular the to {\it free Brownian motion case} as for the {\it non commutative fractional Brownian motion} (abbreged in {\it ncfBm}), this formula and more generally the version of the Clark-Ocone formula will be as in the classical case much harder to obtain since the norm may increase by multiplying with indicators functions (see proposition 3.2 in Léon and Nualart \cite{LNual} or Bender and J.Elliot \cite{BL}).
\end{flushleft}
\begin{lemma}\label{lder}
Let $F\in \mathbb{D}^{1,2}$, and $\mathcal{F}_{\mathcal{A}}$ the previously defined von Neumann subalgebra of $W^*(\left\{S_t,t\geq 0\right\})$, 
then $\tau(F|\mathcal{F}_{\mathcal{A}})\in 
\mathbb{D}^{1,2}$, and:
\begin{equation}
    \nabla_t(\tau(F|\mathcal{F}_{\mathcal{A}}))=\tau\otimes \tau(\nabla_t(F)|\mathcal{F}_{\mathcal{A}}).\mathds{1}_{\mathcal{A}}(t)
\end{equation}
\end{lemma}
\begin{proof}
From the previous lemma, remark that is a consequence that $\tau(F|\mathcal{F}_{\mathcal{A}})\in \mathbb{D}^{1,2}$. Indeed:
\begin{equation}
    \sum_{n=0}^{\infty} n\lVert f_n\mathds{1}_{\mathcal{A}}^{\otimes n}\rVert_{L^2(\mathbb{R}^n_+)}^2\leq \sum_{n=0}^{\infty} n\lVert f_n\rVert_{L^2(\mathbb{R}^n_+)}^2 =\lVert\nabla F\rVert^2_{\mathcal{B}_2}<\infty\nonumber
\end{equation}
since it just multiply each element of the chaotic decomposition by $\mathds{1}_{\mathcal{A}}^{\otimes n}$, and thus does not increase for each $n\geq 0$ the norm $\lVert f_n\rVert_2$.
\bigbreak

Now, we have for $F=\sum_{n=0}^{\infty}I_n(f_n)\in \mathbb{D}^{1,2}$
\begin{eqnarray}
\nabla_t(\tau(F|\mathcal{F}_{\mathcal{A}}))&=&\sum_{n=0}^{\infty} \nabla_t(I_n(f_n\mathds{1}_{\mathcal{A}}^{\otimes n}))\nonumber\\
&=&\sum_{n=0}^{\infty} \sum_{k=1}^n I_{k-1}\otimes I_{n-k}(f_n^{k,t}.\mathds{1}_{\mathcal{A}}^{\otimes (n-1)}).\mathds{1}_{\mathcal{A}}(t)\nonumber\\
&=& \tau\otimes \tau(\nabla_t(F)|\mathcal{F}_{\mathcal{A}}).\mathds{1}_{\mathcal{A}}(t)\nonumber
\end{eqnarray}
 We recall that for almost all $t\geq 0$, $f_n^{k,t}.\mathds{1}_{\mathcal{A}}^{\otimes (n-1)}$ is seen as a function of $L^2(\mathbb{R}^{k-1}_+)\otimes L^2(\mathbb{R}^{n-k}_+)$.
\qed
\end{proof}

\bigbreak
As a direct consequence, we have the free following corollary which gives the {\it local property} of the Malliavin derivative:
\begin{lemma}
Let $F\in\mathbb{D}^{1,2}$, then if $F$ belongs to $\mathcal{F}_{\mathcal{A}}$, then for $t\in \mathcal{A}^c$, we have:
\begin{equation}
\nabla_t(F)=0
\end{equation}
\end{lemma}
\begin{flushleft}
Note also that the we have a commutation relation between a Malliavin differentiable element and the {\it left annihilation operator}. This lemma being particularly useful for its connection with the study of  regularity properties of analytic distribution since it provides a kind of {\it infinite dimensional dualy system} in the sense of Voiculescu defined in section 5 in \cite{V} and which have deep consequences in the study of the regularity properties of the spectral measure on noncommutative polynomials evaluations in self-adjoints indeterminates , i.e. $P(y_1,\ldots,y_n)$. For example, when $y_1,\ldots,y_n$ admits a {\it dual system}, and that $P$ is self-adjoint and also an eigenvalue of the "discrete" Number operator, where $N=\sum_{j=1}^n \partial_j(\:)\sharp y_j$, Charlesworth and Shlyakhtenko proved that the analytic distribution of $P(y_1,\ldots,y_n)$ is absolutely continuous w.r.t the Lebesgue measure (c.f theorem 13 in \cite{CS}).
\end{flushleft}
\begin{lemma}\label{lma9}
Let $n\geq 0$ a fixed integer. Then for all $F\in \mathcal{P}_n$, we have:
\begin{equation}
    [r^*(h),F]=\nabla^h(F)\sharp P_1
\end{equation}
where $P_1$ is the orthogonal projection onto $1\in L^2(\mathcal{SC})$ and $\nabla^h$ is the directional Malliavin derivative: $\nabla^h:=\langle \nabla,h\rangle$ on $\mathbb{D}^{1,2}$.
\end{lemma}

\begin{proof}
Indeed for $(e_i)_{i=1}^{\infty}$ a complete orthonormal system of $L^2_{\mathbb{R}}(\mathbb{R}_+)$, if we denote
\newline
$\Lambda=\left\{a:=(a_n)_{n\geq 1}, a_n\in \mathbb{N}, \exists M, a_n=0, \mbox{for},n>M\right\}$ the integers valued sequences such that all the terms except a finite numbers of them vanish. Then we define for $a\in \Lambda$, with $(U_n)_{n \geq 0}$, the sequence of the Chebychev polynomials of second kind.
\begin{equation}
    \phi_a=\prod_{i=1}^{\infty} U_{a_i}(S(e_i)),\nonumber
\end{equation}
Then,
\begin{equation}
    \left\{\phi_a,a\in \Lambda,\lvert a\rvert=n\right\},\nonumber
\end{equation}
is an orthonormal system in $\mathcal{H}_n$, the $n$-fold Wigner chaos. And, as a consequence, the family:
\begin{equation}
   \left\{\phi_a,a\in \Lambda \right\},
\end{equation}
 is a complete orthonormal system in $L^2(\mathcal{SC})$,
\begin{flushleft}
    Indeed it is easy to prove that (derivation property) for any integer $n\geq 1$, since it is easy to check and well known that: $[r^*(h),S(e_i)]=\langle e_i,h\rangle_{L^2(\mathbb{R}_+)}.P_1$, for all $h\in L^2_{\mathbb{R}}(\mathbb{R}_+)$.
\begin{eqnarray}\label{60}
    [r^*(h),U_n(S(e_i))]&=&  \partial_{S(e_i)}U_n(S(e_i))\sharp [r^*(h),S(e_i)]\\\nonumber
    &=&\sum_{k=1}^n \langle e_i,h\rangle_{L^2(\mathbb{R}_+)} (U_{k-1}(S(e_i))\otimes U_{n-k}(S(e_i)))\sharp P_1\nonumber\\
    &=& \nabla^h(U_n(S(e_i))\sharp P_1,\nonumber
\end{eqnarray}
This last term being exactly the directional Malliavin derivative on the direction $h\in L^2_{\mathbb{R}}(\mathbb{R}_+)$.
\end{flushleft}
Then by linearity and density it extends to the whole space since both $[r^*(h),.]$ and $\nabla^h(.)\sharp P_1$ are derivations and coincide onto this orthonormal basis which spans $\mathcal{SC}_{alg}$ which is dense in $\mathbb{D}^{1,2}$.
\qed
\end{proof}

\begin{flushleft}
Note also that there exists a free Poincaré inequality on the Wigner space (which have never been mentioned or used) and which can be seen as the infinite dimensional analog of the one's proved by Voiculescu in an unpublished note (see Dabrowski \cite{DAB}, lemma 2.2).
\end{flushleft}
\begin{theorem}
Let $F\in \mathbb{D}^{1,2}$, then we have the free Poincaré inequality:
\begin{equation}
    \lVert F-\tau(F)\rVert_2^2\leq \lVert \nabla F\rVert_{\mathcal{B}_2}^2\nonumber
\end{equation}
\end{theorem}
The proof is straightforward via the chaotic decomposition or the free Clark-Ocone-Bismut formula.
\begin{remark}
It turns out that, this inequality gives an important result about the structure of von Neumann algebra generated by a free Brownian motion: it is a factor. Indeed it is well known by the work of Voiculescu that the von Neumann algebra generated by the free Brownian motion, we also use the standard notation $\mathcal{SC}=W^*(\left\{S_t,t\geq 0\right\})\simeq L(\mathbb{F}_{\infty})$ (where the second term is defined as the weak-operator closure of the left regular representation of the non-abelian free group with countably many generator) is a factor since $\mathbb{F}_{\infty}$ of is an {\it i.c.c} (infinite conjugacy class) group. The idea is to prove this (without knowing this isomorphism), as a consequence of the free Poincaré inequality on the Wigner space, as it was inspired to us by the powerful Dabrowski's techniques found in \cite{DAB}, which proved that $W^*(X_1,\ldots,X_n)$ with $X_1,\ldots,X_n$ self-adjoint operators in some tracial $W^*$-probability space is a full $\Pi_1$ factor (without property $\Gamma$) under finite free Fisher information $\Phi^*(X)<\infty$. We will prove our result using the free Malliavin calculus, which can be seen as an infinite dimensional version of the theory of non commutative derivatives over a semicircular field.
\end{remark}
\begin{flushleft}
Before stating the next result, we introduce for the non specialist reader a few notions about von Neumann algebras. 
\end{flushleft}
\begin{definition}
Let $S\subset \mathcal{B}(\mathcal{H})$, by definition its commutant is defined as:
\begin{equation}
    S^{'}=\left\{T \in \mathcal{B}(\mathcal{H}), ST=TS\right\}
, \end{equation}
\end{definition}
\begin{definition}
The center of a von Neumann algebra
$\mathcal{M}$ is defined as the intersection of $\mathcal{M}$ with its commutant $\mathcal{M}^{'}$ :
    \begin{equation}
        \mathcal{Z}(\mathcal{M})=\mathcal{M}\cap\mathcal{M}'=\mathbb{C}1,
    , \end{equation}
\end{definition}
We also remind for the interested reader, that there exists three distinct types of factors (type $I$, type $II$ and type $III$). Here we are in the setting of $II_1$ factors.
\begin{definition}
A $\Pi_1$ factor is an infinite dimensional von Neumann algebra equipped with a (unique) faithful normal tracial state, and which has a trivial center:
\begin{equation}
\mathcal{Z}(\mathcal{M})=\mathbb{C}.1,
\end{equation}
\end{definition}

\begin{theorem}
$\mathbb{D}^{1,2}\cap \mathcal{SC}$ does not contain any central element, i.e. element commuting with all bounded others.
\end{theorem}
\begin{proof}
The basic idea is the following one:
\bigbreak
We denote $(e_i)_{i=1}^{\infty}$ a complete orthonormal system of $L^2_{\mathbb{R}}(\mathbb{R}_+)$,
and the directional Malliavin derivative (introduced by Mai in \cite{Mai}) in the direction $h\in L^2_{\mathbb{R}}(\mathbb{R}_+)$ and valued into the coarse correspondence, defined for $F$ a smooth functional:
\begin{eqnarray}
\nabla^h: \mathcal{SC}_{alg}&\rightarrow& L^2(\mathcal{SC})\otimes L^2(\mathcal{SC})\simeq HS(L^2(\mathcal{SC}))\\\nonumber
F&\mapsto&\langle \nabla F,h\rangle_{L^2(\mathbb{R}_+)}:=\int _{\mathbb{R}_+}\nabla_tF\:h(t)dt
\end{eqnarray}
with domain
\begin{equation}
    D(\nabla^h):=\mathcal{SC}_{alg}=*-\mbox{unital}\left\{S(h), h\in L^2_{\mathbb{R}}(\mathbb{R}_+)\right\}\nonumber
\end{equation} 

Then we can easily show that this is a unbounded closable operator (the closure is still denoted $\nabla^h$) whose domain $D(\nabla^h)$ is defined as the closure of the unital $*$-algebra generated by the free Brownian motion with respect to the directional derivative seminorm:
\begin{equation}
    \lVert Y\rVert_{1,2,h}^2=\lVert Y\rVert_2^2+\lVert \nabla^hY\rVert^2_{L^2(\tau\otimes\tau)}\nonumber
\end{equation}
And note that when $F\in \mathbb{D}^{1,2}$, $\nabla^h F=\langle \nabla F,h\rangle_{L^2(\mathbb{R}_+)} $.
\newline
Now, assume that $F\in \mathcal{Z}(\mathcal{SC})\cap \mathbb{D}^{1,2}$ , and in particular, one has for all $j\in \mathbb{N}$:
\begin{equation}
    [F,S(e_j)]=0\nonumber
\end{equation}
and this implies for any $i\neq j$, that:
\begin{equation}
    0=[\nabla^{e_i}F,S(e_j)]+[F, \nabla^{e_i}S(e_j)]\nonumber
\end{equation}
it is then standard to check that since the family is orthogonal that $\nabla^{e_i}(S(e_j))=\langle e_j,e_i\rangle_{L^2(\mathbb{R})}.1\otimes 1=0$, since $(S(e_i))_{i=1}^{\infty}$ is a (infinite) semicircular system.
\newline
Then it gives us that:
\begin{equation}
     0=[\nabla^{e_i}F,S(e_j)]\nonumber
\end{equation}
Now, one sees $\nabla^{e_i}F\in HS(L^2(\mathcal{SC}))$, an Hilbert-Schmidt operator (and a fortiori a compact one), by the usual identification recalled in the section 2 (equation \ref{HS}) of this paper, it is easily deduced, because $S(e_j)$ is a semicircular operator, it is also a diffuse operator as it as a Lebesgue absolutely continuous spectral measure (since its analytic distribution is absolutely continuous w.r.t the Lebesgue measure), and thus since a compact operator commuting with a diffuse one must be identically zero, it follows without further effort that: 
\begin{equation}
    \nabla^{e_i}F=0
\end{equation}
since $e_i$ was arbitrary, it implies that $\nabla F=0_{\mathcal{B}_2}$, and thus by the free Poincare inequality, we obtain the desired conclusion.

\end{proof}
\qedhere

\begin{flushleft}
This last result allows in fact to the factoriality by using regularisation with the help of resolvents maps associated to the free directional Ornstein-Uhlenbeck operator. However, we need to introduce a powerful lemma of Popa: corollary p.818 in \cite{popa}, refined by Dabrowski lemma 7 in \cite{DAB} which allows to obtain a quantitative estimate in $L^2$ on how a tracial, normal and completely positive map $\phi$ almost fixes the translation of an element $x\in \mathcal{M}$ by a element of the unit ball $y\in (\mathcal{M})_1$, this means evaluate $\lVert \phi(yx)-y\phi(x)\rVert_2, y\in (\mathcal{M})_1, x\in \mathcal{M}$ in terms of how this map $\phi$ almost fixes this same unit ball: $\lVert\phi(y)-y\rVert_2, u\in (\mathcal{M})_1$.
\end{flushleft}
\begin{lemma}(Popa corollary p.818 in \cite{popa} , Dabrowski lemma 7 in \cite{DAB})\label{OZA}
Let $\phi$ a unital tracial completely positive map, then for every $x,y\in \mathcal{M}$ with $\lVert y\rVert\leq 1$, we have:
\begin{equation}
    \lVert \phi(yx)-y\phi(x)\rVert_2\leq \lVert \phi\rVert^{\frac{1}{2}} \lVert x\rVert (2\lVert \phi(y)-y\rVert_2^2+2\lVert \phi(y)-y\rVert_2)^{\frac{1}{2}}
\end{equation}
\end{lemma}
\begin{corollary}\label{thf}
$\mathcal{SC}:=W^*(\left\{S_t,t\geq 0\right\})$, is a factor.
\end{corollary}
\begin{proof}
It only remains to remove the domain assumption, that is to prove that any central element $Z\in \mathcal{Z}(\mathcal{SC})$ is in fact in $\mathbb{D}^{1,2}$. An essential tool to be able to do this, is the closability of the free Malliavin gradient. Indeed, if we are able to find a smooth (in $\mathbb{D}^{1,2}$) net $(Z_{\alpha})_{\alpha\geq 0}$ converging in $L^2$ towards $Z$ and which also belongs to the center, the proof is obtained. 
\newline
This is essentially done through regularisation in the Malliavin calculus sense.
\newline

Indeed, let's  denote $\delta^h$, the adjoint of the directional Malliavin derivative in the direction $h\in L^2_{\mathbb{R}}(\mathbb{R}_+)$ which is a densely defined operator:
\begin{eqnarray}
\delta^h: dom(\delta^h) \subset L^2(\mathcal{SC})\otimes L^2(\mathcal{SC})&\rightarrow& L^2(\mathcal{SC})\nonumber\\
U&\mapsto& \delta(U\sharp (1\otimes h\otimes 1))
\end{eqnarray}
Moreover, one knows from Mai (c.f lemma 5.2 (b) in \cite{Mai}) that when restricted to real vectors $h\in L^2_{\mathbb{R}}(\mathbb{R}_+)$, $\nabla^h$ induces a {\it real closable derivation}. This enables us to directly use results which can be found in Cipriani and Sauvageot \cite{CipS} or the introduction by Dabrowski in \cite{DAB}.

\bigbreak
In this way, we can define the directional number operator (Laplacian) in the direction $h\in L^2_{\mathbb{R}}(\mathbb{R}_+)$, which is defined as the square of the directional Malliavin derivative:
\begin{eqnarray}
N^h: dom(N^h)\subset L^2(\mathcal{SC})&\rightarrow& L^2(\mathcal{SC})\nonumber\\
F&\mapsto& \delta^h(\nabla^hF)
\end{eqnarray}
and see that it is the generator of a {\it completely} Dirichlet form ($N^h\otimes I_n$ is also a Dirichlet form on $M_n(SC)$ for all $n\geq 0$). 
We denote $\phi_t^h:=e^{-N^ht}$ the semigroup of generator $N^h$.
\bigbreak
Now consider for each $h\in L^2_{\mathbb{R}}(\mathbb{R}_+)$, the associated resolvent maps defined for $\alpha>0$ to each directional Laplacian, and which will be denoted as 
\begin{equation}
    \eta_{\alpha,h}:=\alpha(\alpha+N^h)^{-1}
\end{equation}
and \begin{equation}
\zeta_{\alpha,h}:=\eta_{\alpha,h}^{\frac{1}{2}}
\end{equation}
which are moreover unital, tracial ($\tau\circ \eta_{\alpha,h}=\tau$), positive and completely positive and moreover contractions on $L^2(\mathcal{SC})$.
\newline
With the following properties: $Range(\eta_{\alpha,h})=dom(N^h)\subset D(\nabla^h)$ and $Range(\zeta_{\alpha,h})=dom(\nabla^h)$ and moreover that $\nabla^h\circ \zeta_{\alpha,h}$ is bounded.
\newline
Note also that we have (its an easy computation) $\lVert Z-\zeta_{\alpha,h}(Z)\rVert_2\underset{\alpha \rightarrow \infty}{\rightarrow 0}$ for all $h\in L^2_{\mathbb{R}}(\mathbb{R}_+)$.
\newline
Now by definition, we have for $Y\in \mathbb{C}\langle X(e_j)_{j\neq i}\rangle$, $\nabla^{e_i}(Y)=0$.
Thus it is straightforward from the lemma \ref{OZA} that
$\zeta_{\alpha,e_i}(FX(e_j))=\zeta_{\alpha,e_i}(F)X(e_j)$ and $\zeta_{\alpha,e_i}(X(e_j)F)=X(e_j)\zeta_{\alpha,e_i}(F)$.
\bigbreak
Thus, if we denote for $\alpha>0$ and $h\in L^2_{\mathbb{R}}(\mathbb{R}_+)$,
$\nabla^h_{\alpha}:=\alpha^{-\frac{1}{2}}\nabla^h\circ \zeta_{\alpha,h}$.
\newline
We then have for $Y\in \mathbb{C}\langle X(e_j)_{j\neq i}\rangle$ with in particular $\nabla^{e_i}(Y)=0$:
\begin{equation}
\nabla^{e_i}_{\alpha}([Z,Y])=[\nabla^{e_i}_{\alpha}(Z),Y]
\end{equation}
Thus if $Z\in \mathcal{Z}(\mathcal{SC})$, we have:
\begin{equation}
[\nabla^{e_i}_{\alpha}(Z),Y]=0
\end{equation}
By closability, we then have (by using the same arguments as below: an Hilbert-Schmidt operator commuting with a diffuse one must be identically zero) that $\nabla^{e_i}_{\alpha}(Z)=0$.
\newline
Since $\lVert Z-\zeta_{\alpha,{e_i}}(Z)\rVert_2\underset{\alpha \rightarrow \infty}{\rightarrow 0}$, we get that $Z\in D(\nabla^{e_i})$ and $\nabla^{e_i}Z=0$.

It is then straightforward to check for any $h\in L^2_{\mathbb{R}}(\mathbb{R}_+)$, (apply the same analysis by computing
\newline
$[\nabla^{h}_{\alpha}(Z),Y]=0$ with $Y$ which belongs to the $\mathbb{C}\langle X(h^{'})_{h'\in \left\{h\right\}^{\perp}}\rangle$ (in the orthogonal complement of $h$ in $L^2_{\mathbb{R}}(\mathbb{R}_+)$), that is $\nabla^h Z=0$, and we can also check easily that the map $h\mapsto \nabla^h Z$ is linear and continuous.
\newline
Thus $Z\in \mathbb{D}^{1,2}$ and $\nabla Z=0$, thus the free Poincaré inequality on the Wigner space concludes.
\qed
\end{proof}

\begin{remark}
It is also well known by the breakthrough result of Voiculescu using free entropy techniques (theorem 5.3 in \cite{V}), that this von Neumann algebra does not have regular {\it DHSA: diffuse hyperfinite abelian subalgebra}, and a fortiori no Cartan subalgebra as in a $\Pi_1$ factor, regular  Maximal abelian subalgebras ({\it MASA's}) are also {\it DHSA}. 
In their important work, Dabrowski and Ioana \cite{Di}, these authors have been able to prove various indecomposability results such as non-$L^2$ rigidity and absence of Cartan subalgebras in presence of a real closable derivation into the coarse bimodule whose domain contain a weakly dense $*$-subalgebra $\mathcal{M}_0$ and various additional conditions (but not at all restrictive) such as existence of a non-amenability set and that moreover the derivation is algebraic. Indeed Dabrowski in \cite{Dab10} proved that under the assumptions of Lipschitz conjugate variables the von Neumann algebra $M=W^*(X_1,\ldots,X_n)$ generated by $X_1,\ldots,X_n$ admits a {\it free dilation}: a deformation into $M*L(\mathbb{F}_{\infty})$, then important findings of Ioana in \cite{ioa} and Popa's deformation/rigidity theory allows them to deduce absence of Cartan subalgebras. However in the infinite dimensional case, we have not been able to prove that $L(\mathbb{F}_{\infty})$ does not have Cartan subalgebras by using again this infinite dimensional differential calculus on the Wigner space. We leave this question to an interested reader.
\end{remark}
We are for now only able to prove an infinite dimensional version in the semicircular case of a result of Shlyakhtenko \cite{S04} which gives condition on algebraic derivations to give rise via exponentiation to a one parameter semigroup of $M*L(\mathbb{Z})$ . Thus our results can be seen as an elementary infinite dimensional case of the general result of Shlyakhtenko, since in our context,  of each $S(h), h\in L^2_{\mathbb{R}}(\mathbb{R}_+)$ satisfies the most standard integration by parts, and can also be seen as an infinite system of Lipschitz conjugate variables (idea introduced by Dabrowski in \cite{Dab14}):
\begin{equation}
    \tau(FS(h))=\tau\otimes\tau(\langle \nabla F,h\rangle_{L^2(\mathbb{R}_+)})
\end{equation}
for all $h\in L^2_{\mathbb{R}}(\mathbb{R})$ and $F\in \mathcal{SC}_{alg}$ and that for almost all $t\geq 0$, $\nabla_t(S(h))=h(t).1\otimes 1\in \mathcal{SC}\otimes \mathcal{SC}$.
\begin{proposition}
    Let $M=W^*(\left\{S_t,t\geq 0\right\}\simeq L(\mathbb{F}_{\infty})$ and $M_{alg}$ be respectively the von Neumann algebra and the $*$-unital algebra generated by the free Brownian motion $S=\left\{S(h):=\int_{\mathbb{R}_+} h(s)dS(s), h\in L^2_{\mathbb{R}}(\mathbb{R}_+)\right\}$ and consider in this way, $\nabla^S$ and $\delta^S$ be respectively the free Malliavin derivative and the Skorohod integral with respect to the semicircular process $S$.
\newline
Let also $X$ denotes an independent copy of the semicircular process $S$, which generates also $L(\mathbb{F}_{\infty})$. Then there exists a one parameter semigroup of automorphism $(\alpha_t)_{t\geq 0}$ of $M*W^*(\left\{X(h),h\in L^2_{\mathbb{R}}(\mathbb{R_+})\right\})\simeq M*L(\mathbb{F}_{\infty})$, and such that:
\begin{eqnarray}
   \frac{d}{dt}\bigg|_{t=0}\alpha_t(F)&=& X_{\nabla^S(F)} , \forall F\in M_{alg},\nonumber\\
        \frac{d}{dt}\bigg|_{t=0}\alpha_t(X(h))&=&-\delta^S(h.1\otimes 1):=-S(h), \forall h\in L^2_{\mathbb{R}}(\mathbb{R}_+),\nonumber,
\end{eqnarray}
where for $F=S(e_1)\ldots S(e_n)$ with $e_1,\ldots,e_n \in L^2_{\mathbb{R}}(\mathbb{R}_+)$, we define
\newline
$X_{\nabla^S(F)}:=\sum_{k=1}^n S(e_1)\ldots S(e_{k-1})X(e_k)S(e_{k+1})\ldots S(e_n)$, and we then extend it linearly.
\end{proposition}
\begin{proof}
Indeed, denote $X=\left\{X(h), h\in L^2_{\mathbb{R}}(\mathbb{R}_+)\right\}$ an independent copy of the semicircular isonormal process $S=\left\{S(h), h\in L^2_{\mathbb{R}}(\mathbb{R}_+)\right\}$, 
\newline
Then, it suffices to consider for $t\geq 0$, 
\begin{eqnarray}
    \alpha_t(S(h))&=&\cos(t)S(h)+\sin(t)X(h)\nonumber\\ 
     \alpha_t(X(h))&=&-\sin(t)S(h)+\cos(t)X(h)\nonumber
\end{eqnarray}
which is well defined and clearly an automorphism of $M*W^*(\left\{S_t,t\geq 0\right\})$, then by differentiating, we get:
\begin{eqnarray}
    \frac{d}{dt}\bigg|_{|t=0}\alpha(S(h))&=&X(h)\nonumber\\
    \frac{d}{dt}\bigg|_{|t=0}\alpha_t(X(h))&=&-S(h)=-\delta^S(h.1\otimes 1)\nonumber
\end{eqnarray}
and by using the chain rule of the free Malliavin derivative its is elementary to check that the result holds true for $F=S(e_1)\ldots S(e_n)$.
\end{proof}
\begin{flushleft}
We also provide here an interesting characterization about the triviality of projections in $\mathbb{D}^{1,2}$ which can be seen as a corollary of Guionnet and Shlyakhtenko, theorem 4.2 in \cite{GS} which states that $C^*(\left\{S_t,t\geq 0\right\})$ is projectionless, and which ensures that the distribution of any self-adjoint element of this $C^*$-algebra has a connected support. Our proposition thus shows that projections are not regular functionals in the free Malliavin calculus sense.
\begin{flushleft}

Indeed, in his celebrated monograph, Nualart (see proposition 1.2.6 in \cite{Nual}) proved an interesting proposition about existence of Malliavin differentiable idempotent in Gaussian spaces, that is to say, for any Borel set set $A$ with respect to the $\sigma$-algebra generated by a Gaussian isonormal process, the following Bernoulli random variable: $\mathds{1}_{A}$ belongs to $\mathbb{D}^{1,2}$
if and if only $\mathbb{P}(A)\in \left\{0,1\right\}$.
We propositionose here to show the free counterpart of this "{\it $0-1$ law}".
\end{flushleft}

\begin{proposition}
$\mathbb{D}^{1,2}\cap W^*(\left\{S_t,t\geq 0\right\})$ does not contain non trivial projections, that is
for $p=p^*=p^2\in W^*(\left\{S_t,t\geq 0\right\})$ a projection, then $p \in \mathbb{D}^{1,2}$ if and if only $p\in \left\{0,1\right\}$.
\end{proposition}
\begin{proof}
Take a smooth approximation $\phi\in \mathcal{I}_2(f)$ of $x\mapsto x^2$ which coincides inside the support of $p$ (compactly supported on the real line since it is a bounded self-adjoint operator), $supp(p)\subset [-\rho(p),\rho(p)]$ ($\rho$ denote the spectral radius), where \begin{equation}
    \mathcal{I}_2(f)=\left\{f:\mathbb{R}\rightarrow \mathbb{C}, f(x)=\int e^{ix}\mu(dy),\int_{\mathbb{R}}\lvert y\rvert^2\lvert\mu\rvert(dy)<\infty,\right\}
\end{equation}
Then by the chain rule (note here that the assumption over $\phi$ is necessary as the chain rule for random variable with possibly infinite chaotic expansion is only-satisfied for this subset of {\it operator-Lipschitz functions}), we get for almost all $t\geq 0$:\begin{equation}\label{p6}
    \nabla_t(p)=\nabla_t (p^2)=p.\nabla_t(p)+\nabla_t(p).p\nonumber
\end{equation}
First remark by using compression on both side by "$p$", it gives:
\begin{eqnarray}
p.\nabla_t(p).p&=&p^2.\nabla_t(p).p+p\nabla_t.p^2\nonumber\\
&=&2p.\nabla_t(p).p\nonumber
\end{eqnarray}
Thus we have:
\begin{equation}\label{p1}
    p.\nabla_t(p).p=0
    \end{equation}
and in particular that:
\begin{equation}
    (p.\nabla_t(p).p)\sharp (\nabla_t(p))^*=0
    \end{equation}
Which easily implies (by faithfulness and traciality since $p$ is a projection) that at least one of the following two terms is zero:
\begin{equation}
p.\nabla_t(p)=0\quad\mathrm{or}\quad \nabla_t(p).p=0\nonumber
\end{equation}
Let's suppose that the first one is zero:
\newline
Then we get, 
\begin{equation}
\nabla_t(p)=\nabla_t(p).p\nonumber
\end{equation}
By multiplying by $p$, on the left side, we then have from \ref{p1}:
\begin{equation}
p.\nabla_t(p)=p.\nabla_t(p).p=0\nonumber
\end{equation}
The other possibility can be dealt with in the same way.
\bigbreak
Thus in any cases by injecting in \ref{p6}, we have $\nabla(p)=0$ and by the free Poincaré inequality it gives that $p=\tau(p).1$, and since $p$ is an idempotent, we get that:
\begin{equation}
\tau(p)=\tau(p)^2
\end{equation}
and finally that $p\in\left\{0,1\right\}$.
\qed
\end{proof}
\end{flushleft}
\begin{remark}
It is important to notice that we have shown a little more since we didn't use the {\it self-adjoint condition} (orthogonality): there is no bounded Malliavin differentiable idempotent. It would be very interesting to prove or disprove an analog of this result for $W^*(X_1,\ldots,X_n)$ where $X_1,\ldots,X_n$ are self-adjoint operators in $(\mathcal{M},\tau)$ a tracial $W^*$-probability space with for example finite Fisher information: $\Phi^*(X)<\infty$. This would require to have a {\it chain rule} for free difference quotients which go beyond polynomial calculus (here it is a particular feature of the {Wigner} space and that we are dealing with functionals of a semicircular system) and thus be able to consider more general classes of functions.
\end{remark}

\section{Higher-order Malliavin derivatives}
In this section, we pursue the analysis of Biane and Speicher in \cite{BS} about free Malliavin calculus by defining higher order (semicircular) {\it Sobolev-Wigner} spaces and we study the action of the Malliavin derivative onto finite Wigner chaos. We obtain in this way the free counterparts of classical important results on the Wiener space. We first recall some definitions about biprocesses and their associated $L^p$-spaces defined as the completion of simple adapted biprocesses. We then focus and introduce some notations in the case of $\mathcal{H}_{\mathbb{R}}=L^2_{\mathbb{R}}({\mathbb{R}_+})$ which will be very useful to lighten the exposure. Note that, up to much heavier notations and definitions it is also possible to extend our definitions to the general setting of "{\it abstract Wigner spaces}".
\begin{definition}
Recall now that a simple biprocess, is a piecewise constant map $t\mapsto U_t$ indexed by positive time such that for all $t\geq 0$, $U_t\in \mathcal{SC}\odot \mathcal{SC}$ (the algebraic tensor product, i.e. without completion), and $U_t=0$ for $t$ large enough.
\end{definition}
\begin{definition} 
A simple biprocess is called adapted if for all $t\geq 0$, $U_t\in \mathcal{A}_t\odot \mathcal{A}_t$.
\end{definition}
\begin{definition}
The (complex) space of simple biprocesses is endowed with the family of norms, defined for all $1\leq p\leq \infty$ as:
\begin{equation}
\lVert U\rVert_{\mathcal{B}_p}=\left(\int_{0}^{\infty}\lVert U_t\rVert^2_{L^p(\tau\otimes \tau^{op})}dt\right)^{\frac{1}{2}}
\end{equation}
\end{definition}
The completion will be denoted as $\mathcal{B}_p$ and the restriction to the closed subspace of $L^p$-integrable adapted biprocesses $\mathcal{B}_p^{a}$.
\begin{flushleft}
When $p=2$, $\mathcal{B}_2$ is an Hilbert space associated with the inner product:
\begin{equation}
\langle U,V\rangle_{\mathcal{B}_2}=\int_0^{\infty} \langle U_t,V_t\rangle_{L^2(\tau\otimes \tau^{op})}dt
\end{equation}
\end{flushleft}

\begin{definition}
We denote for $1\leq p\leq \infty$, and $k,n\geq 1$ positive integers, 
\newline
$\mathbb{M}_p^{k,n}:=L^2(\mathbb{R}^k_+,L^p(\mathcal{SC}^{\otimes n}))$, the random field of $L^p(\mathcal{SC}^{\otimes n})$ valued noncommutative random variables equipped with the following semi-norm:
\begin{equation}
\lVert U\rVert_{\mathbb{M}_p^{k,n}}=\left(\int_{\mathbb{R}^k_+}\lVert U_{t_1,\ldots,t_k}\rVert_{L^p(\tau^{\otimes n})}^2dt_1\ldots dt_k\right)^{\frac{1}{2}}\nonumber
\end{equation}
\end{definition}
\begin{remark}
We can see in this way that the class of $L^p$-biprocesses $\mathcal{B}_p$ is simply $\mathbb{M}_{p}^{1,2}$. Note that we adopt this more general definition which can be used independently to study "{\it free random fields}". Indeed we can restrict ourselves to the study of these spaces with $n=k+1$, which is sufficient for our purpose and specified in the following.
\end{remark}

\begin{definition}
We define for each $p\geq 1$, the iterated free Malliavin derivative of order "$n$" denoted $\nabla^n$ by setting for any $F\in \mathcal{SC}_{alg}$,
$\nabla^0=Id$, $\nabla^{1}=\nabla$, and for $n>1$,
\begin{eqnarray}
    \nabla^{n} &:& \mathcal{SC}_{alg}\rightarrow L^2(\mathbb{R}_+^n,\mathcal{SC}^{\odot (n+1)}))\nonumber\\
    &&F\mapsto
    (\nabla_{t_1,\ldots,t_n}^n F)_{t_1\ldots,t_n\geq 0}:=(id^{\otimes (n-1)}\otimes \nabla_{t_n})\circ \ldots \circ(id \otimes \nabla_{t_2})\circ \nabla_{t_1}
    (F)\nonumber
\end{eqnarray}
or equivalently, we have by recursion for almost all $t_1,\ldots,t_n\geq 0$,
\begin{equation}
\nabla_{t_1,\ldots,t_n}^n F=(id^{\otimes (n-1)}\otimes \nabla_{t_n})\nabla^n_{t_1,\ldots,t_{n-1}}F
\end{equation}
\end{definition}
\bigbreak
Sometimes, we will however prefer to use this following definition as the free Malliavin derivative will be shown to be an almost everywhere coassociative derivation which justify the fact to multiply by the coefficient $n!$ to get some kind of symmetrized expression. Note that it will also provide exact free counterparts of the formulas obtained on the Wiener space ({\it Stroock's formula, chaotic characterization of the Sobolev-Watanabe spaces}...).
\begin{definition}
We set $D^n, n\geq 1$, the symmetrized variant of the free Malliavin derivative by setting for $F\in \mathcal{SC}_{alg}$ :
\begin{equation}
D^nF:=n!\nabla^nF
\end{equation}
\end{definition}

\begin{remark}
It is straightforward to check that both of these map are well-defined on $\mathcal{SC}_{alg}$ and also does not depend on the choice of the representation. Note also that the free Malliavin derivative is an almost sure {\it coassociative derivation} as formulated in the next proposition. It is really important to insist on it as it provides infinite dimensional Dabrowski inequalities (section 2.2 in \cite{Dab14}): the boundedness of the directional Malliavin derivative: i.e. $L^2$-bounds of $(id \otimes \tau )\circ \nabla^h$ (see Mai remark 5.7 of \cite{Mai} for applications).
\end{remark}
\begin{proposition}\label{coas}
The free Malliavin derivative is an almost-everywhere coassociative derivation in the following sense:
\bigbreak
For all $F\in \mathcal{\mathcal{SC}}_{alg}$, and for almost all $s,t\geq 0$:
\begin{equation}
(\nabla_{t}\otimes id)\nabla_{s}F=(id\otimes \nabla_{s})\nabla_{t}F.
\end{equation}
\end{proposition}
\begin{proof}
Indeed it is sufficient to show the equality for $F=S(e_1)\ldots S(e_n)$ with $n\geq 1$ (since $\mathcal{SC}_{alg}$ consists of finite sum of this type of elements), where $(e_i)_{i=1}^{\infty}$ is an orthonormal basis of $L^2_{\mathbb{R}}(\mathbb{R}_+)$. 
\newline
Then, we have:
\begin{equation}
    \nabla_{s}F=\sum_{1\leq i_2\leq n} e_{i_2}(s).S(e_1)\ldots S(e_{i_2-1}) \otimes S(e_{i_2+1})\ldots S(e_n)\nonumber
\end{equation}
And thus:
\begin{equation}
    (\nabla_{t}\otimes id)\nabla_{s}F=\sum_{1\leq i_1<i_2\leq n} e_{i_1}(t)e_{i_2}(s).S(e_1)\ldots S(e_{i_1-1})\otimes S(e_{i_1+1}) \ldots S(e_{i_2-1})  \otimes S(e_{i_2+1})\ldots S(e_n)\nonumber
\end{equation}
We can compute the other term to get:
\begin{equation}
     \nabla_{t}F=\sum_{1\leq i_1\leq n} e_{i_1}(t).S(e_1)\ldots S(e_{i_1-1}) \otimes S(e_{i_1+1})\ldots S(e_n)\nonumber
\end{equation}
and we have:
\begin{equation}
    (id\otimes \nabla_{s})\nabla_{t}F=\sum_{1\leq i_1<i_2\leq n} e_{i_1}(t)e_{i_2}(s).S(e_1)\ldots S(e_{i_1-1})\otimes S(e_{i_1+1}) \ldots S(e_{i_2-1})  \otimes S(e_{i_2+1})\ldots S(e_n)\nonumber
\end{equation}
\bigbreak
Thus both expression coincide onto $\mathcal{SC}_{alg}$ and it concludes.
\qed
\end{proof}

\begin{remark}
Note that we might drop the subscript denoting the Hilbert space $L^2(\mathbb{R}_+)$ when we use inner products or pairings $\langle \nabla F, h\rangle$ since the context is clear here.
\end{remark}
\begin{definition}
We denote for $U\in L^2(\mathbb{R}_+^n, SC^{\odot (n+1)})\subset \mathbb{M}_{\infty}^{n,n+1}$, and $h_1,\ldots,h_n\in L^2(\mathbb{R}_+)$ the linear extension of the pairing
\begin{equation}
\langle U, h_1\otimes \ldots \otimes h_n\rangle:=\int_{\mathbb{R}_+^n}U_{t_1,\ldots,t_n} \overline{h_1(t_1)}\ldots \overline{h_n(t_n)}dt_1\ldots dt_n
\end{equation}
\end{definition}
Note that by Cauchy-Schwartz inequality we have:
\begin{equation}
\lVert \langle U,h_1\otimes \ldots \otimes h_n\rangle \rVert_{L^{\infty}(SC^{\otimes (n+1)})}\leq \lVert U\rVert_{\mathbb{M}_{\infty}^{n,n+1}}\rVert \lVert h_1\otimes\ldots\otimes h_n\rVert_{L^2(\mathbb{R}_+^n)},
\end{equation}
thus the pairing could be extended continuously to a pairing between $\mathbb{M}^{n,n+1}_{\infty}$ and $L^2(\mathbb{R}_+^n)$.
\begin{flushleft}
This enables us to provide a first interesting generalized integration-by-parts at the second-order. We will generalize this fact to higher-orders in the sequel.
\end{flushleft}
\begin{proposition}
Let $F\in \mathcal{SC}_{alg}$, $h_1,h_2\in L^2_{\mathbb{R}}(\mathbb{R}_+)$, then
\begin{equation}\label{ip1}
\tau^{\otimes 3}(\langle \nabla^2 F,h_1\otimes h_2\rangle)=\tau(FI_2(h_2\otimes h_1))
\end{equation}
\end{proposition}
\begin{proof}
There is at least two distinct way to reach the conclusion. The first one is more sophisticated and relies on proving the closability of the linear operator $(id\otimes \nabla^{h_2})\circ \nabla^{h_1}$.
\newline
Indeed, rewriting the equation \ref{ip1},
it involves to prove that for $F\in \mathcal{SC}_{alg}$ we have:
\begin{equation}
\tau^{\otimes 3}((id\otimes \nabla^{h_2})\nabla^{h_1}F)=\tau(FI_2(h_2\otimes h_1))
\end{equation}
Said otherwise, it amounts to prove that, $=\left((id\otimes \nabla^{h_2})\circ \nabla^{h_1}\right)^*(1\otimes 1\otimes 1):=(\nabla^{h_1})^*((\nabla^{h_2})^*(1\otimes 1)\otimes 1)$ exists in $L^2(\mathcal{SC})$.
That is to say in the language of noncommutative derivatives that a kind a second-order conjugate variable exists in $L^2(\mathcal{SC})$, we show in fact more since the second-order conjugate variables will be bounded (even Lipschitz).
Indeed, one knows from Mai's (recalled in the proof of theorem \ref{thf}) that in this case $\nabla^{h_1}$ is a {\it real closable derivation}, 
\bigbreak
Then a simple computations show that if this term exists and is in fact equal to:
\begin{eqnarray}
(\nabla^{h_1})^*((\nabla^{h_2})^*(1\otimes 1)\otimes 1)&:=&\delta^{h_1}(\delta^{h_2}(1\otimes 1)\otimes 1)\nonumber\\
&=&\delta^{h_1}( S(h_2)\otimes 1)\nonumber\\
&=& S(h_2)S(h_1)-m_1\circ(id\otimes \tau\otimes id)(\langle \tilde{\nabla}  (S(h_2)\otimes 1),h_1\rangle)\nonumber\\
&=&S(h_2)S(h_1)-\langle h_2,h_1\rangle\nonumber\\
&:=& I_2(h_2\otimes h_1)
\end{eqnarray}
where in the last line, we used the infinite-dimensional version of the Voiculescu formula's \ref{ppV}.
\bigbreak
We propose another proof based on a direct computation.
\newline
We know from proposition \ref{coas},
that for $F=S(e_1)\ldots S(e_n)$ for $e_1,\ldots, e_n \in L^2_{\mathbb{R}}(\mathbb{R}_+)$ (not necessarily a complete orthonormal system), for some $n\geq 1$.
\begin{eqnarray}
\langle\nabla^2F,h_1\otimes h_2\rangle=
\sum_{1\leq i_1<i_2\leq n} \langle e_{i_1},h_1\rangle \langle e_{i_2},h_2\rangle.S(e_1)\ldots S(e_{i_1-1})\otimes S(e_{i_1+1}) \ldots S(e_{i_2-1})  \nonumber\\
\otimes S(e_{i_2+1})\ldots S(e_n)\nonumber
\end{eqnarray}
and by applying the trace:
\begin{eqnarray}
\label{ipp2}
\tau^{\otimes 3}(\langle\nabla^2F,h_1\otimes h_2\rangle)=
\sum_{1\leq i_1<i_2\leq n} \langle e_{i_1},h_1\rangle \langle e_{i_2},h_2\rangle.\tau(S(e_1)\ldots S(e_{i_1-1})) \nonumber\\
\tau(S(e_{i_1+1}) \ldots S(e_{i_2-1})) \tau(S(e_{i_2+1})\ldots S(e_n))\nonumber
\end{eqnarray}

Now by applying two times the Wick identity (note we used the following convention: when $i_1=n+1$, $e_{n+1}=h_2$), we get
\begin{align*}
&\tau(FX(h_2)X(h_1))\nonumber\\
&=\sum_{1\leq i_1\leq n+1} \langle e_{i_1},h_1\rangle \tau(S(e_1)\ldots X(e_{i_1-1}))\tau(S(e_{i_1+1})\ldots S(e_{n+1}))\nonumber\\
&=\sum_{1\leq i_1\leq n} \langle e_{i_1},h_1\rangle \tau(S(e_1)\ldots S(e_{i_1-1}))\tau(S(e_{i_1+1})\ldots S(h_1))+\langle h_2,h_1\rangle \tau(F)\nonumber\\
&= \sum_{1\leq i_1<i_2\leq n} \langle e_{i_1},h_1\rangle \langle e_{i_2},h_2\rangle \tau(S(e_1)\ldots X(e_{i_1-1}))\tau(S(e_{i_1+1})\ldots S(e_{i_2-1}))\ldots \tau(S(e_{i_2+1})\ldots S(e_{n}))\nonumber\\
&+\langle h_2,h_1\rangle\tau(F)
\end{align*}
Thus by comparing with the expression \ref{ipp2}, we get:
\begin{equation}
\tau(FI_2(h_2\otimes h_1)):=\tau(F[S(h_2)S(h_1)-\langle h_2,h_1\rangle])=\tau^{\otimes 3}(\langle \nabla^2 F,h_1\otimes h_2\rangle)
\end{equation}
which concludes.
\qed
\end{proof}
Note that the definition of the gradient which acts onto biprocesses is different from our first one and useful to deduce the Heisenberg commutation relation (c.f section 5.4) in \cite{BS}. Indeed this last one is our second "symmetrized" version of the free Malliavin derivative which acts onto biprocesses as following for almost all $s,t\geq 0$, $\tilde{\nabla}_t(A_s\otimes B_s):= (\nabla_{t}\otimes id+id\otimes \nabla_{t_1})(A_s\otimes B_s)=\nabla_tA_s\otimes B_s+A_s\otimes \nabla_t B_s$ (via the coassociativity). We will emphasize this use of this version of the gradient when needed, by specifying it as $\tilde{\nabla}$. Indeed, as we will see later, if we used it, it would simply and only complexify some formulas by adding some combinatorial constant in these, as in the Wigner space it is well known that this type of formulas are much simpler than on the Wigner space.
\begin{definition}
We define the non commutative semicircular Sobolev 
\newline
spaces $\mathbb{D}^{k,p}$ for any $p\geq 1$ and any natural number $k\geq 1$ by setting
on $\mathcal{SC}_{alg}$ the following family of seminorm:
\begin{equation}
\lVert F\rVert_{k,p}=\left[\lVert F\rVert_p^p+\sum_{j=1}^k\lVert \nabla^k
F\rVert^p_{\mathbb{M}_p^{k,k+1}}\right]^{\frac{1}{p}}
\end{equation}
\end{definition}
As we have introduced a {\it symmetrized} variant of the free Malliavin derivative, one can thus consider in the same vein their associated semicircular Sobolev spaces.
\begin{definition}
We define the symmetrized semicircular Sobolev spaces $\mathbb{D}^{k,p,\sigma}$ for any $p\geq 1$ and any natural number $k\geq 1$ by setting
on $\mathcal{SC}_{alg}$ the following family of seminorm:
\begin{equation}
\lVert F\rVert_{k,p,\sigma}=\left[\lVert F\rVert_p^p+\sum_{j=1}^k\lVert D^kF\rVert^p_{\mathbb{M}_p^{k,k+1}}\right]^{\frac{1}{p}}
\end{equation}
\end{definition}
With these definitions, one can then obtain the closability of these iterated gradients, as well as compatility relations that are expected when using classical Sobolev spaces.
\begin{proposition}
For any $p\geq 1$, and $k\geq 1$, an integer, $\nabla^k$ is a closable operator from $L^p(\mathcal{SC})$ to $\mathbb{M}_{p}^{n,n+1}$ and the domain of its closure is denoted $\mathbb{D}^{k,p}$. For $p=\infty$, it is a closable operator from $L^{\infty}(\mathcal{SC})$ to $\mathbb{M}_{\infty}^{n,n+1}$ for the weak-operator topology.
\end{proposition}
\begin{proof}
The proof is similar as in the case $k=1$ of Biane and Speicher, but involve much more notations.
\newline
Indeed to show the closability of $\nabla^n, n\geq 1$, consider $(F_n)_n \in \mathcal{SC}_{alg}$, a sequence of cylindrical functionals converging towards zero, i.e. $F_n\rightarrow 0$ (in $L^p(\mathcal{SC}), p\geq 1$, resp. weakly) and such that $\nabla^k F_n\rightarrow U\in \mathbb{M}_{p}^{k,k+1}$, then we have to prove that $U=0$ and it is sufficient to show
that:
\begin{equation}
\tau^{\otimes(k+1)}\left(Z^{(k+1)}\sharp\langle U,h_1\otimes \ldots\otimes h_n\rangle\right)=0
\end{equation}
for any choice of $h_1,\ldots,h_n\in L^2_{\mathbb{R}}(\mathbb{R}_+)$, and $Z_1,\ldots,Z_{k+1}\in \mathcal{SC}_{alg}$ where we denoted $Z^{(k+1)}=Z_1\otimes\ldots\otimes Z_{k+1}$ and where $\sharp$ denotes the linear extension of the multiplication onto $\mathcal{SC}^{\otimes k}\otimes \mathcal{SC}^{op}$ (i.e. $(a_1\otimes \ldots\otimes a_{k+1})\sharp (b_1\otimes\ldots\otimes b_{k+1})=a_1b_1\otimes a_2b_2\otimes\ldots \otimes b_{k+1}a_{k+1}$, $a_i,b_i \in \mathcal{SC}$, for $i=1,\ldots,k+1$).
Then by using recursively the integration by parts formula \ref{80}, one can then achieve the proof.

\end{proof}

\qedhere
\begin{theorem}
Replacing $\nabla$, by $D$ in the previous theorem, the conclusion also holds true when we consider the symmetrised semicircular Sobolev spaces whose closure are denoted ${\mathbb{D}}^{k,p,\sigma}, k,p\geq 1$, and also on $L^{\infty}(\mathcal{SC})$ for the weak operator topology.
\end{theorem}

The following proposition ensures that these non commutative semicirculars Sobolev spaces behaves well with respect to compatibility relations.
\begin{proposition}
This family of seminorms satisfies the following propositionerties:
\begin{enumerate}
\item Almost-Monotonicity: $\exists\: c_{k,p,q}>0$, $\lVert F\rVert_{k,p}\leq c_{k,p,q}\lVert F\rVert_{j,q}$ for all $F\in \mathcal{SC}_{alg}$ and $k\leq j$, $p\leq q$.
    \item Compatibility: Let $k,j\geq 1$ positive integers and $p,q\geq 1$, then if $(F_n)_{n\geq 0}$ is a sequence of smooth cyclindricals functionals, $\lVert F\rVert_{k,p}\rightarrow 0$, and such that $(F_n)_{n\geq 0}$ is a Cauchy sequence for $\lVert .\rVert_{j,q}$, then by closability of $\nabla^j$ on $\mathcal{SC}_{alg}$, we have that $\lVert F_n\rVert_{j,q}\rightarrow 0$.
\end{enumerate}
The statement also holds true when replacing the seminorms by the ones constructed with $D$.
\end{proposition}
\begin{proof}
This is an easy consequence of the closability of the Malliavin gradient $\nabla^i, D^i, i\geq 1$, on $\mathcal{SC}_{alg}$ and non-commutative Holder inequalities.
\qed
\end{proof}
\qedhere
\begin{flushleft}
We can in this way, consider the adjoint of the $p$-th Malliavin derivative denoted as the free Skorokhod integral $\delta^p$ of order $p\in\mathbbm{N}_+$. We can also show (as expected) that we can recover the definition of multiple Wigner integral as Skorokhod integral of {\it deterministic multiprocesses} and show that they coincide as in the classical case (see Nourdin and Peccati section 2.7 in \cite{NP-book}).
\end{flushleft}
\begin{definition}
We define the free Skorohod integral of order $p\geq 1$ (a positive integer), as the adjoint of the free Malliavin derivative of order "$p$", the derivative being view as an unbounded closable operator from $\mathbb{D}^{p,2}$ into $\mathbb{M}_{2}^{p,p+1}$. We will denote it as $\delta^p$ being an operator from $dom(\delta^p)\subset \mathbb{M}_{2}^{p+1}$ onto $L^2(\mathcal{SC})$.
\newline
which in particular satisfies the following duality relation:
\newline
Let $u\in dom(\delta^p)$, then for all $F\in \mathcal{SC}_{alg}$, there exists a unique element denoted $\delta^p(u)$ such that:
\begin{equation}
\langle \nabla^p
 F,u\rangle_{\mathbb{M}_2^{p,p+1}}=\langle F,\delta^p(u)\rangle_{L^2(\mathcal{SC})},
\end{equation}
\end{definition}
\begin{remark}
When $p=1$, by choosing in \ref{79}, $u=h.1\otimes 1$ for $h\in L^2(\mathbb{R}_+)$, we get that
\newline
$L^2(\mathbb{R}_+).1\otimes 1:=\left\{h.1\otimes 1, h\in L^2(\mathbb{R}_+)\right\}\subset dom(\delta)$, then it is not difficult to prove that
\newline
$(L^2(\mathbb{R}_+))^{\otimes p}.1^{\otimes (p+1)}:=\left\{h.1^{\otimes (p+1)}, h\in (L^2(\mathbb{R}_+))^{\otimes p}\simeq L^2(\mathbb{R}_+^p)\right\}\subset dom(\delta^p)$.
\newline
This fact will allows us in particular to recover that multiple Wigner integral of order "$p$" are exactly given by free Skorohod integral of order "$p$" of {\it deterministic} multiprocesses. This idea was first introduce in the classical setting by Nualart and Zakai in \cite{nualZ}, and turns out to be an important fact, as various important statement on the classical case can be shown in an easier way by using this definition (eg. Hypercontractivity in the the finite Wiener chaos which are a consequence of the so-called Meyer formulas, see e.g. section 2.7 of Nourdin and Peccati \cite{NP-book}).
\end{remark}
\begin{proposition}
For any integer $p\geq 1$, and $f\in L^2(\mathbb{R}_+^p)$
\begin{equation}
    \delta^p(f.1^{\otimes (p+1)})=I_p(f)
\end{equation}
\end{proposition}
\begin{proof}
Let us prove by induction (over "$p$") that the result hold true for $f=h^{\otimes p}$, with $h\in L^2_{\mathbb{R}}(\mathbb{R}_+)$, with $\lVert h\rVert_{L^2(\mathbb{R}_+)}=1$, and thus that 
we have, \begin{equation}
    U_p(S(h))=\delta^p(h^{\otimes p}.1^{\otimes (p+1)})
\end{equation}
For $p=1$, it is clear, since $U_1=X$ and $\delta(h.1\otimes 1)=S(h)$.
\newline
Now, suppose that the result holds true for $1,2\ldots,p$, 
\newline
And note that:
$\delta^{p+1}(u)=\delta((\delta^{p}\otimes id_{L^2(\mathbb{R}_+,L^2(\mathcal{SC}))}(u))$, which implies that:
\begin{align*}
    &\delta^{(p+1)}(h^{\otimes (p+1)}.1^{\otimes (p+2)})\nonumber\\
    &=\delta (h. \delta^p(h^{\otimes p}.1^{\otimes (p+1)})\otimes 1)\nonumber\\
    &=\delta(h.I_p(h^{\otimes p})\otimes 1)\nonumber\\
    &=I_p(h^{\otimes p})S(h)-m_1\circ(id\otimes\tau\otimes id)(\langle \tilde{\nabla} (I_p(h^{\otimes p})\otimes 1),h\rangle_{L^2(\mathbb{R}_+)})\nonumber\\
     &=I_p(h^{\otimes p})S(h)-m_1\circ(id\otimes\tau\otimes id)\left( \langle h,h\rangle_{L^2(\mathbb{R}_+)} \sum_{k=1}^{p} I_{k-1}\otimes I_{p-k}(h^{\otimes (p-1)})\otimes 1\right) \nonumber\\
      &=I_p(h^{\otimes p})S(h)-m_1\left( I_{p-1}(h^{\otimes (p-1)})\otimes 1\right) \nonumber\\
    &=U_p(S(h))S(h)-U_{p-1}(S(h))\nonumber\\
    &=U_{p+1}(S(h))\nonumber
\end{align*}
where in the third line we used the induction hypothesis,in the fourth line we used the {\it infinite dimensional Voiculescu's formulas} (c.f proposition \ref{ppV}), and in the fifth line, the explicit action of the Malliavin derivative (of order 1) onto Wigner integrals (c.f proposition \ref{pp3}).
\bigbreak
Note that the passage from the fourth line to the fifth one is because Wigner integrals of order greater than or equal to one are centered, thus only term in the sum which contributes to the sum is when $k=p$, which gives just a constant (a multiple Wigner integral of order 0) and we also used the hypothesis that $\lVert h\rVert_{L^2(\mathbb{R}_+)}=1$,
Now, the general conclusion will follow by linearity and density of such elementary tensors, i.e. the linear span of the complexification $U_{\mathbb{C}}$ where $U=\left\{h^{\otimes p},h\in L^2_{\mathbb{R}}(\mathbb{R}_+)/\lVert h\rVert_{L^2(\mathbb{R}_+)}=1\right\}$ is dense in $L^2(\mathbb{R}^p_+)$ as well as the closability of $\delta^p$.
\qed
\end{proof}
\begin{remark}
The reader which knows the notion of "{\it higher-order conjugate variables}" in the sense of Voiculescu (c.f definition 3.1 in \cite{V}) might notice that we have in fact recover in a simpler way the proposition 3.8 in \cite{V} which says that the $p$-th order conjugate variable of a standard semicircular variable $S$ is given by the $p$-th Tchebychev polynomial in this semicircular variable $S$. This fact which  can be rewritten in a more sophisticated way as $\partial^{(p)^*}(1^{\otimes (p+1)})=U_p(S)$, where $\partial^{(p)}$ denotes the $p$-th iterated free-difference quotient with respect to the semicircular variable $S$ (which we remind is an unbounded and densely defined closable operator for any $p\in \mathbb{N}^*$), is also the free counterpart of the result, $\delta^p(1)=H_p$ where $\delta$ denotes the divergence operator on the Gaussian space $L^2(\gamma)$ where $d\gamma(x)=\frac{1}{\sqrt{2\pi}}e^{-x^2/2}dx$ is the standard Gaussian measure and with $H_p$ denoting the $p$-th Hermite polynomial (see Nourdin and Peccati, section 1.4 in \cite{NP-book} for a complete and didactic exposure).
\end{remark}
\begin{definition}
We define for each $p\geq 1$, the following intersection of Sobolev-Wigner spaces (we only consider here the symmetrized spaces since the first are included in these ones) which is the space of $L^p$-{\it test functionals} and defined as:

\begin{eqnarray}
    \mathbb{D}^{\infty,p,\sigma}:=\bigcap_{k=1}\mathbb{D}^{k,p,\sigma}
\end{eqnarray}
\end{definition}
\begin{definition}
We set the space of {\it test Wigner functionals} as:
\begin{equation}
    \mathbb{D}^{\infty,\sigma}:=\bigcap_{p\geq 1}\bigcap_{k=1}\mathbb{D}^{k,p,\sigma}
\end{equation}

\end{definition}
\begin{remark}
A standard application of the noncommutative Holder inequalities shows that this space is in fact an unital $*$-algebra. It is also important to notice that finite sum of multiple Wigner integrals always belongs to $\mathbb{D}^{\infty,\sigma}$.
\end{remark}
\begin{flushleft}
We can now level up the action of the free Malliavin derivative of higher-order onto some fixed Wigner chaos.
\end{flushleft}
\begin{proposition}\label{proposition9}
For $p\geq 1$, a positive integer, the action of $\nabla^p$ onto homogeneous Wigner chaos of order greater than $p$, that is for all $n\geq p$ and $F=I_n(f)$ with $f\in L^2(\mathbb{R}_+^n)$ is given by:
\begin{equation}
    \nabla^p_{t_1,\ldots,t_p}\left(I_n(f)\right)=\sum_{1\leq i_1<\ldots <i_p\leq n}I_{i_1-1}\otimes I_{i_2-i_1-1}\otimes \ldots\otimes I_{n-i_p}(f_{t_1,\ldots,t_p}^{i_1,\ldots,i_n})
\end{equation}
 and for almost all $t_1,\ldots,t_p\geq 0$, and every $1\leq i_1<\ldots <i_p\leq n$ we have:
\begin{align*}
    &f(s_1,\ldots,s_{i_1-1},\underbrace {t_{1}}_{i_1},s_{i_1+1},\ldots,s_{i_2-i_1},\underbrace{t_{p-1}}_{i_{p-1}},\ldots, \underbrace{t_{p}}_{i_p},s_{i_p+1},\ldots,s_n)\nonumber\\
    &=f_{t_1,\ldots,t_p}^{i_1,\ldots,i_p}(s_1,\ldots,s_{i_1-1},s_{i_1+1},\ldots,s_{i_2-1},\ldots s_n)\nonumber
\end{align*}
and where we regard for almost all $t_1\ldots,t_p$ fixed, and for all $i_1\ldots i_p$,
\newline
$f_{t_1,\ldots,t_p}^{i_1,\ldots,i_n}$ as an element of $L^2(\mathbb{R}^{i_1-1}_+)\otimes \ldots\otimes L^2(\mathbb{R}^{n-i_p}_+)$.
\end{proposition}
When $p>n$, the action of the free Malliavin derivative is trivial:
\newline
$\nabla^p_{| \mathcal{P}_n}=0$.
\begin{proof}
To explain the formula, first remark and
note that it is easily deduced onto elementary multiple Wigner integrals $U_n(S(e_j))=I_n(e_j^{\otimes n})$ for $n\geq 1$, where as previously $(e_j)_{j=1}^{\infty}$ is a complete orthonormal system of $L^2_{\mathbb{R}}(\mathbb{R}_+)$.
\bigbreak
Indeed, we have at the first order:
\begin{equation}
    \nabla_{t_1}(U_n(S(e_j)))=e_{j}(t_1)\sum_{1\leq i_1\leq n}U_{i_1-1}(S(e_j))\otimes U_{n-i_1}(S(e_j))
\end{equation}
and thus (and also remark that the higher degree of Tchebychev polynomials considered is "$n-2$" since the derivative of $U_{n-1}(S(e_j))\otimes 1$, $1\otimes U_{n-1}(S(e_j))$ which correspond to the term in the sum $k\in\left\{1,n\right\}$ trivially vanishes) from the almost everywhere coassociativity relation, we get (in the second equality, we use the change of summation $i_2\rightarrow i_2+i_1$)
\begin{eqnarray}
    \nabla^2_{t_1,t_2}(U_n(S(e_j)))&=& e_j^{\otimes 2}(t_1,t_2)\sum_{1\leq i_1\leq n}\sum_{1\leq i_2\leq n-i_1} 
    U_{i_1-1}(S(e_j))\otimes U_{i_2-1}(S(e_j))\otimes U_{n-i_1-i_2}(S(e_j))\nonumber\\
    &=&e_j^{\otimes 2}(t_1,t_2)\sum_{1<i_1<i_2\leq n}U_{i_1-1}(S(e_j))\otimes U_{i_2-i_1-1}(S(e_j))\otimes U_{n-i_2}(S(e_j))\nonumber
\end{eqnarray}
The general case can also be deduced by checking the relation for multiple Wigner integrals expressed in terns of elementary tensors.
\bigbreak
for $p=1$, it reduces to the action of the free Malliavin derivative onto multiple Wigner integrals, which is true by the proposition \ref{pp3}.
\bigbreak

By density and linearity, since $\nabla^p$ is closable, it suffices to prove the result for some $I_n(f)$ with $f$ an elementary (real valued) tensor of the form:
\newline
take $f=f_1\otimes\ldots\otimes f_n$ with each $f_i \in L^2_{\mathbb{R}}(\mathbb{R}_+)$.
\bigbreak
We obtain at first order:
\begin{equation}
    \nabla_{t_1}(I_n(f))=\sum_{1\leq i_1\leq n}f_{i_1}(t_1)I_{i_1-1}(f_1\otimes \ldots\otimes f_{i_1-1})\otimes I_{n-i_1}(f_{i_1+1}\otimes\ldots \otimes f_n)\nonumber
\end{equation}
And then by applying iteratively $id \otimes \nabla_{t_j}$, we get:
\begin{eqnarray}
    \nabla_{t_1,t_2}^2(I_n(f))=\sum_{1\leq i_1< i_2\leq n}f_{i_1}({t_1})f_{i_2}({t_2}) I_{i_1-1}(f_1\otimes \ldots \otimes f_{i_1-1})\otimes I_{i_2-i_1-1}(f_{i_1+1}\otimes \ldots \otimes f_{i_2-1})\nonumber\\
   \otimes I_{n-i_2}(f_{i_2+1}\otimes  \ldots \otimes f_n)\nonumber
\end{eqnarray}
Similar computations for higher order Malliavin derivatives lead to the result.
\qed
\end{proof}
\begin{remark}\label{3}
For sake of simplicity, we will often write the tensor product 
\newline
$I_n^{\otimes (i_1,\ldots,i_p)}:=I_{i_1-1}\otimes\ldots\ldots\otimes I_{n-i_p}$ to shorten the notations.
\end{remark}

\begin{flushleft}
As in the classical case, we can characterize the Wigner functionals which belongs to these $L^2$-Sobolev spaces via their chaotic decomposition. Here we only consider the symmetrized version as it will provide the same combinatoric appearing in the classical chaotic characterization (i.e. the number of derangement) .
\end{flushleft}
\begin{proposition}
Let's $F=\sum_{n=0}^{\infty}I_n(f_n)\in L^2(\mathcal{SC})$,
then for any integer $p\geq 1$, $F$ belongs to $\mathbb{D}^{p,2,\sigma}$, if and if only:
\begin{equation}
\sum_{n=p}^{\infty}n(n-1)\ldots(n-p+1)\lVert f_n\rVert^2_{L^2(\mathbbm{R}^n_+)}<\infty
\end{equation}
and in this case:
\begin{equation}\lVert   D^pF\rVert_{\mathbb{M}_2^{p,p+1}}=\sum_{n=p}^{\infty}n(n-1)\ldots(n-p+1)\lVert f_n\rVert^2_{L^2(\mathbbm{R}^n_+)}<\infty
\end{equation}
\end{proposition}
\begin{proof}
 The reader may also notice that as expected of a derivation of order $p\geq 1$, the elements of the finite Wigner chaos of order stricly less than "$p$" vanishes. It is then only useful to consider the case $F=I_n(f)$, $f\in L^2(\mathbb{R}_+^n), n\geq p$.

\bigbreak
From an approximation argument and the action of the free Malliavin derivative of order $p\geq 1$, onto homogeneous Wigner chaos, it is then sufficient to remark that given any $1\leq i_1\ldots <i_p\leq n$, and $t_1\ldots,t_p$ fixed, we have from the Wigner-Ito isometry for multiple Wigner-Ito integral which is obviously extended to the tensor case, that the $L^2$-norm given by $\tau^{\otimes (p+1)}$ is:
\begin{equation}
    \bigg\lVert I_{i_1-1}\otimes I_{i_2-i_1-1}\otimes \ldots\otimes I_{n-i_p}(f_{t_1,\ldots,t_p}^{i_1,\ldots,i_n})\bigg\rVert_{L^2(\tau^{\otimes (p+1)})}^2=\lVert f^{i_1,\ldots,i_n}_{t_1,\ldots,t_p}\rVert_{L^2(\mathbb{R}_+^{n-p})}^2
\end{equation}
where the function $f_{t_1,\ldots,t_p}^{i_1,\ldots,i_n}\in L^2(\mathbb{R}_+^{n-p})$ is seen here as a square integrable function in $n-p$ variables.
\newline
By integrating over $\mathbb{R}_+^p$, for each $i_1,\ldots, i_p$ fixed it gives the same contribution which is exactly the $L^2$ norm of the function $f$.
\begin{equation}
     \int_{\mathbb{R}^p_+}\bigg\lVert I_{i_1-1}\otimes I_{i_2-i_1-1}\otimes \ldots\otimes I_{n-i_p}(f_{t_1,\ldots,t_p}^{i_1,\ldots,i_n})\bigg\rVert_{L^2(\tau^{\otimes (p+1)})}^2dt_1\ldots dt_p=\lVert f\rVert_{L^2(\mathbb{R}^n_+)}^2.
\end{equation}
\newline
We are then left to evaluate the cardinal of $\sharp\left\{1\leq i_1<\ldots <i_p\leq n\right\}$, which is equal to $\binom{n}{p}$.
\newline
Now it suffice to remark that $p!\binom{n}{p}=n(n-1)\ldots(n-p+1)$
which concludes.
\qed
\end{proof}
As a easy consequence of the two previous propositions, we have:
\begin{proposition}
Let $F\in \mathbb{D}^{2,p,\sigma}$, if $\nabla^p(F)=0$, then $F\in \mathcal{P}_{p-1}$.
\end{proposition}
\begin{flushleft}
Now we will explicit the main theorem of this section which is the free analog of the well known "Stroock's" formula on the Wiener space first proved by Stroock in \cite{stroock}, and which explicitly gives the chaotic decomposition of an infinitely smooth $L^2$-functional in terms of its iterated Malliavin derivatives.

\end{flushleft}
\begin{theorem}(Free Stroock's formula\label{Stro})
\newline
Let $F\in {\mathbbm{D}}^{\infty,2,\sigma}$, with chaotic expansion $F=\sum_{n=0}^{\infty} I_n(f_n)$, then for all $n\geq 0$ and almost all (in the sense of the Lebesgue measure) $t_1,\ldots,t_n\geq 0$:
\begin{equation}
    f_n(t_1,\ldots,t_n)=\tau^{\otimes ({n+1})}\left(\nabla_{t_1,\ldots,t_n}^nF\right)
\end{equation}
Or the symmetrized version
\begin{equation}
    f_n(t_1,\ldots,t_n)=\frac{1}{n!}\tau^{\otimes ({n+1})}\left(D_{t_1,\ldots,t_n}^nF\right)
\end{equation}
\end{theorem}
\begin{proof}
By linearity and density, since $\nabla^p, p\geq 1$ is a densely defined unbounded closable operator, it suffices to show the result for multiple Wigner-Ito integrals of any order $n\geq 1$:
\newline
By the proposition 6, 
\begin{equation}
 \nabla^p_{t_1,\ldots,t_p}\left(I_n(f)\right)=\sum_{1\leq i_1<\ldots <i_p\leq n}I_{i_1-1}\otimes I_{i_2-i_1-1}\otimes \ldots\otimes I_{n-i_p}(f_{t_1,\ldots,t_p}^{i_1,\ldots,i_n})
\end{equation}
then applying $\tau^{\otimes (p+1)}$ for all $p<n$, is is easily seen that there is at least for any $1\leq i_1<\ldots<i_p\leq n$, a non zero integer in the following sequence $i_1-1,i_2-i_1-1,\ldots,n-i_p$, thus since Wigner integral are centered, all the terms vanishes.
\newline
Now for $p>n$ the higher order Malliavin derivatives trivially vanishes as expected: $\nabla^p_{|
{\mathcal{P}_n}}=0$.
\newline
And for the only (non trivial) contributing case, that is $p=n$, it is straightforward to check that the only term is (for almost all $t_1,\ldots,t_n\geq 0$)
\begin{equation}
\nabla^n_{t_1,\ldots,t_n}(I_n(f))=f(t_1,\ldots,t_n).1^{\otimes{(n+1)}}
\end{equation}
and the result follows.
\qed
\end{proof}
As a straightforward application of our previous findings, we also have a generalized commutation relation between the Ornstein-Uhlenbeck semigroup and the higher-order Malliavin derivatives.
\begin{proposition}
Let $k$ a positive integer and $F\in \mathbb{D}^{k,2}$, then we have the following relation which holds true for every $t>0$:
\begin{equation}
    \nabla^k P_tF=e^{-kt}P_t^{\otimes (k+1)}\nabla^k F
\end{equation}
\end{proposition}
\section{Variance formulas on the Wigner space}
This {\it free Stroock's formula} also gives a simple and free counterpart of variances expansion in terms of infinite series for infinitely smooth Wigner functionals. The interested reader might consult the book of Nourdin and Peccati \cite{NP-book} section 1.5 chapter 1 for analogous results in the classical case. In fact this formula
is closely linked with the free number operator $N:=-L$, the opposite of the free Ornstein-Uhlenbeck operator.
\bigbreak
We first begin with a variance formula involving the {\it free Ornstein-Uhlenbeck semigroup}:
\begin{lemma}
For all $F,G \in \mathbb{D}^{1,2}$,
\begin{equation}
cov(F,G):=\tau\bigg[(F-\tau(F))(G-\tau(G))\bigg]=\int_0^{\infty} e^{-t}\langle P_t^{\otimes 2}(\nabla F),\nabla G^*\rangle_{\mathcal{B}_2}dt
\end{equation}
\end{lemma}
\begin{proof}
Let's denote the (continuous) function $h(t):=-\tau(P_tFG)$, then $h(0)=\tau(FG)$ and by ergodicity of $(P_t)_{t\geq 0}$ which means that the subalgebra of {\it fixed-points}:
\begin{equation}
\mathcal{N}:=\left\{F \in \mathcal{SC}, P_tF=F, \mbox{for\:all}\:t\geq 0\right\}=\mathbb{C}.1
\end{equation}
is trivial (which is easy to prove), implies that $\underset{t\rightarrow \infty}{\lim} h(t)=-\tau(F)\tau(G)$.
\newline
Thus one has by the fundamental rule of calculus:
\begin{eqnarray}
cov(F,G)&=&\int_0^{\infty} -\frac{d}{dt}h(t)dt\nonumber\\
&=&\int_0^{\infty} \tau\bigg(-\frac{dP_tF}{dt}G\bigg)dt
\nonumber\\
&=&\int_0^{\infty}\tau(L(P_tF)G)dt\nonumber\\
&=& \int_0^{\infty} \langle \nabla P_tF,\nabla G^*\rangle_{\mathcal{B}_2}dt\nonumber\\
&=&\int_0^{\infty} e^{-t} \langle P_t^{\otimes 2}(\nabla F),\nabla G^*\rangle_{\mathcal{B}_2}dt
\end{eqnarray}
\end{proof}
\begin{proposition}
    Let $F\in \mathbb{D}^{\infty,2,\sigma}$, then:
    \begin{equation}
    var(F):=\lVert F-\tau(F)\rVert_2^2=\sum_{n=1}^{\infty}\lVert\tau^{\otimes (n+1)}(\nabla^nF)\rVert_{L^2(\mathbb{R}^n_+)}
    \end{equation}
where $\tau^{\otimes(n+1)}(\nabla^n F)$ is understand point-wise amost everywhere in the sense of the product Lebesgue measure.
\end{proposition}
\begin{proof}
The statement is an easy consequence of the previous Stroock's formula combined with the Wigner-Ito isometry.
\qed
\end{proof}
\bigbreak
Note that these last formulas are not the only variances estimates in terms of free Malliavin operators, there is also the well known (as it is the free counterpart) of the integration by parts on the Wiener space involving $\nabla, L^{-1}$, the last operator being the pseudo-inverse of the free Ornstein-Uhlenbeck operator. Surprisingly, this last one doesn't seem to be an appropriate variance formula on the Wigner space especially when dealing with its combination with the free Stein's method. Cébron (lemma $3.9$ in \cite{C}) discovered a much more powerful formula which seems to traduce better the properties of this differential calculus. We first recall its formula, and we provide generalized versions which involve the higher order free Malliavin derivatives.

\begin{lemma}
For all $A,B \in \mathcal{P}_n$, such as $\tau(A)=0$ or $\tau(B)=0$:
\begin{equation}
\tau(AB)=\tau\left(\int_{\mathbb{R}_+} id\otimes\tau(\nabla_t A).(\tau\otimes id(\nabla_t B))dt\right),
\end{equation}
\end{lemma}
We can now state a much more interesting variance formula which can be seen as the generalization of the formula discovered by Cébron and which is linked to the higher order free Malliavin derivatives.
\begin{lemma}(Generalized Cébron formulas)
Let $b\geq a\geq p>1$ be positive integers, then for all $A,B \in \bigoplus_{k=a}^b\mathcal{H}_k$ :
\begin{equation}
\tau(AB)=\tau\left(\int_{\mathbb{R}^p_+} (id\otimes\tau^{\otimes p})(\nabla_{t_p,\ldots,t_1}^pA).(\tau^{\otimes p}\otimes id)(\nabla_{t_1,\ldots,t_p}^pB)dt_1\ldots dt_p\right),
\end{equation}
\end{lemma}
\begin{proof}
First, it is important to notice that in the term $(id\otimes \tau^{\otimes p})(\nabla^p_{t_p,\ldots,t_1}$), that the variables $t_1,\ldots,t_p$ are taken in reverse order (the reader familiar with usual and free Malliavin calculus can notice that to be able to interchange the variables, it is necessary to assume the {\it fully-symmetry} of the multiple Wigner integrals considered).
\newline
By linearity of the free Malliavin derivatives, it is sufficient to show the proposition for elementary multiple Wigner integrals.
\newline
Moreover since elements of $\mathcal{H}_k, k\geq 1$ are centered (and it's not an hypothesis) we have that $\tau(A)=\tau(B)=0$.
\newline
Thus it suffice to consider for $a\leq n,m\leq b$ with $n,m\geq p$ and
\newline
$A=I_n(f),B=I_m(f)$ with $f=f_1\otimes\ldots\otimes f_n$ and $g=g_1\otimes\ldots\otimes g_m$ and each $f_i,g_j\in L^2_{\mathbb{R}}(\mathbb{R}_+)$ for $i=1,\ldots,n$ and $j=1,\ldots,m$.
\bigbreak
Then we have:
\begin{equation}
\nabla^p_{t_p,\ldots,t_1}(I_n(f))=\sum_{1\leq i_1<\ldots<i_p\leq n}f_{i_1}(t_p)\ldots f_{i_p}(t_1) I_{i_1-1}(f_1\otimes \ldots\otimes f_{i_1-1})\otimes \ldots\otimes I_{n-i_p}(f_{i_p+1}\otimes \ldots\otimes f_{n})
\end{equation}
Now applying $id\otimes \tau^{\otimes p}$, we are left with:
\begin{equation}
(id\otimes \tau^{\otimes p})(\nabla^p_{t_p,\ldots,t_1}A)=f_n(t_1)\ldots f_{n-p+1}(t_p)I_{n-p}(f_{1}\otimes \ldots \otimes f_{n-p})
\end{equation}
Indeed since Wigner integrals of order $k\geq 1$ are centered, all the terms vanishes expect when $i_p=n,i_{p-1}=n-1,\ldots, i_2-i_1-1=0$ which implies recursively that $i_1=n-p+1$.
\newline
Similarly, we have:
\begin{equation}
(\tau^{\otimes p}\otimes id)(\nabla^p_{t_1,\ldots,t_p}B)=g_1(t_1)\ldots g_p(t_p)I_{m-p}(g_{p+1}\otimes \ldots \otimes g_m)
\end{equation}
Now,
\begin{align*}
&\tau\left(\int_{\mathbb{R}^p_+}(id\otimes \tau^{\otimes p})(\nabla_{t_p,\ldots,t_1}^pA).(\tau^{\otimes p}\otimes id)(\nabla^p_{t_1,\ldots,t_p}B)dt_1\ldots dt_p\right)\nonumber\\
&=\int_{\mathbb{R}^p_+}f_{n}(t_1)\ldots f_{n-p+1}(t_p)g_1(t_1)\ldots g_p(t_p)\tau\left(I_{n-p}(f_1\otimes \ldots\otimes f_{n-p}).I_{m-p}(g_{p+1}\otimes\ldots \otimes g_m)\right)dt_1\ldots dt_p
\nonumber\\
&=\int_{\mathbb{R}^p_+}f_n(t_1)\ldots f_{n-p+1}(t_p)g_1(t_1)\ldots g_p(t_p) \delta_{n-p,m-p} (f_{1}\otimes \ldots\otimes f_{n-p})\stackrel{p}{\frown}(g_{p+1}\otimes\ldots\otimes g_m)dt_1\ldots dt_p\nonumber\\
&=\delta_{n,m}.f\stackrel{n}{\frown}g
\nonumber\\
&=\tau(AB)\nonumber\\
\end{align*}
where here we denote $\delta$ as the {\it delta Kronecker symbol}.
\qed
\end{proof}
\begin{remark}
This lemma will be particularly useful for its connection with the free-Malliavin Stein method, especially when dealing with higher-order derivatives of noncommutative polynomials, this we will be investigated in another forthcoming paper.
\end{remark}
\section{Product formula on the Wigner space}
We recalled in the first section, the product formula proved by Biane and Speicher in \cite{BS}, which provides the chaotic decomposition (linearization formula) of the the product of two multiple Wigner integrals. In this section, we will prove that this formula is as in the classical case (idea which first appeared in the work of \"Ust\"unel \cite{ust}) just a consequence of a kind of Leibniz formula (idea introduced by Voiculescu in discussion preceding proposition 4.5 in \cite{V}) and also satisfied by the free Malliavin gradient. This proof is thus more in spirit with its connection to Malliavin calculus.
\begin{flushleft}
Before giving the proof of this formula, we remind to the reader, that it is straightforward to check (see for example Voiculescu page 205 in \cite{V}) that:
\end{flushleft}
\begin{lemma}
For any $F,G$ in $\mathcal{\mathcal{SC}}_{alg}$, then, for almost all $t_1,\ldots,t_n\geq 0$:
\begin{equation}
    \nabla^n_{t_1,\ldots,t_n}(FG)=\sum_{k=0}^n (id^{\otimes k}\otimes m_1\otimes id^{\otimes (n-k)})(\nabla^{k}_{t_{1},\ldots,t_{n-k}}(F)\otimes \nabla^{n-k}_{t_{n-k+1},\ldots,t_{n}}(G))
\end{equation}
where by convention when $k\in\left\{0,n\right\}$, $\nabla^0(F)\otimes \nabla^n(G):=F.\nabla^n(G)$ and $\nabla^n(F)\otimes \nabla^0(G):=\nabla^n(F).G$ are respectively the multiplication of the first and last leg.
\end{lemma}
\begin{proof}
Thus we fix $F,G\in \mathcal{\mathcal{SC}}_{alg}$. We first explain how it works on the first few orders, $n=1,2,3$. The general result being just an iterative application of the Malliavin gradient.
\newline
Indeed, for $n=1$ it is clear, since it's just reduces to the Leibniz rule of the free Malliavin derivative onto $\mathcal{\mathcal{SC}}_{alg}$.
\begin{equation}
\nabla_{t_1}(FG)=\nabla_{t_1}(F).G+\nabla_{t_1}(G).F\nonumber
\end{equation}
Which gives at the second order $n=2$, for $t_1,t_2\geq 0$,
\begin{eqnarray}
\nabla_{t_1,t_2}^2(FG)&:=&(id\otimes \nabla_{t_2})(\nabla_{t_1}(F).G+F.\nabla_{t_1}(G))\nonumber\\
&=&\nabla_{t_1,t_2}^2(F).G+(id\otimes m_1\otimes id)(\nabla_{t_1}(F)\otimes \nabla_{t_2}(G))+F.\nabla_{t_1,t_2}^2(G)
\end{eqnarray}
Then at the third order, we have:
\begin{eqnarray}
\nabla_{t_1,t_2,t_3}^3(FG)&:=&(id \otimes \nabla_{t_3})(\nabla_{t_1,t_2}^2(F).G+(id\otimes m_1\otimes id)(\nabla_{t_1}(F)\otimes \nabla_{t_2}(G))+F.\nabla_{t_1,t_2}^2(G))\nonumber\\
&=&\nabla_{t_1,t_2,t_3}^3(F).G+(id^{\otimes 2}\otimes m_1\otimes id)(\nabla_{t_1,t_2}^2(F)\otimes \nabla_{t_3}(G))\nonumber\\
&+&(id\otimes m_1\otimes id^2)(\nabla_{t_1}(F)\otimes \nabla_{t_2,t_3}^2(G))+F.\nabla_{t_1,t_2,t_3}^2(G)\nonumber
\end{eqnarray}
Applying the same scheme for higher order yields the result.
\end{proof}
\qedhere
\bigbreak
We can now get back to proof of \ref{pprodui}.
\begin{proof}(proposition \ref{pprodui})
First remark that for all $n,m\geq 0$, $I_n(f),I_n(g)\in \mathbb{D}^{\infty,4,\sigma}$, thus $I_n(f)I_m(g)\in \mathbb{D}^{\infty,2,\sigma}$.
\bigbreak
For sake of clarity and because it suffices to consider $f=f_1\otimes \ldots \otimes f_n$ and $g=g_1\otimes \ldots \otimes g_m$ elementary tensors with $f_i,g_j\in L^2_{\mathbb{R}}(\mathbb{R}_+)$ for all $1\leq i\leq n,1\leq j\leq m$, we only give details about this case, the general result follows by approximation.
\newline
Then, we know from the {\it free Stroock's formula}, that the chaotic decomposition is given by:
\begin{equation}\label{149}
I_n(f)I_m(g)=\sum_{p=0}^{\infty} I_p\left(\tau^{\otimes (p+1)}[\nabla^p(I_n(f)I_m(g))]\right)
\end{equation}
Now, we know by the previous lemma that:
\begin{equation}\label{pp}
\nabla^p(I_n(f)I_m(g))=\sum_{k=0}^p (id^{\otimes k}\otimes m_1\otimes id^{\otimes (p-k)})(\nabla^{k}(I_n(f))\otimes \nabla^{p-k}(I_m(g))
\end{equation}
and thus by using the explicit action of $\nabla$ onto Wigner chaos, we get (note the conventions used in remark \ref{3}):
\begin{align*}\label{pp}
&\nabla^p(I_n(f)I_m(g))\nonumber\\
&=\sum_{k=0\vee(p-m)}^{p\wedge n} \sum_{\substack{1\leq i_1<\ldots <i_k\leq n\\
1\leq j_1<\ldots <j_{p-k}\leq m\\}}(id^{\otimes k}\otimes m_1\otimes id^{\otimes (p-k)})(I_n^{i_1,\ldots,i_{k}}(f)\otimes I_m^{j_1,\ldots,j_{p-k}}(g))\nonumber
\end{align*}
Since Wigner-Ito integrals are centered, by applying $\tau^{\otimes (p+1)}$, we must have to get non-zero terms:
\newline
$i_1-1=\ldots = i_{k}-i_{k-1}-1=j_2-j_1-1\ldots =j_{p-k}-j_{p-k-1}-1=m-j_{p-k}=0$, which thus gives recursively that:
\newline
$i_1=1,\ldots, i_{k-1}=k-1,i_k=k$ and $j_1=m+k-p+1,\ldots,j_{p-k-1}=m-1, j_{p-k}=m$.
\bigbreak
One can see that the "{\it central term}" (the one on which the operator $m_1$ applies) being thus given by the trace of the product of the multiple integrals $I_{n-k}(f_{k+1}\otimes \ldots \otimes f_{n})$ and $I_{m+k-p}(g_1\otimes \ldots \otimes g_{m+k-p})$,
is non zero by the Wigner-Ito isometry if and if only $k=(n-m+p)/2$, thus $k$ and $n+m$ have the same parity and $\lvert n-m\rvert\leq p\leq n+m$.
\newline
Thus there is only one non zero term in the equation \ref{pp} given by $k=(n-m+p)/2$, which gives $n-k=(n+m-p)/2$ and $m+k-p=(n+m-p)/2$ the associated kernel being thus given by:
\begin{align*}
& \tau\bigg(I_{(n+m-p)/2}(f_{(n-m-p+2)/2}\otimes \ldots \otimes f_{n})I_{(n+m-p)/2}(g_1\otimes \ldots \otimes g_{(n+m-p)/2})\bigg).f_1\otimes \ldots\otimes f_{(n-m+p)/2}\nonumber\\ &\otimes g_{(n+m-p+2)/2}\otimes \ldots\otimes g_n\nonumber\\
&=\langle g_1\otimes \ldots\otimes g_{(n+m-p)/2},f_n\otimes \ldots \otimes f_{(n-m+p+2)/2}\rangle_{L^2(\mathbb{R}_+^{(n+m-p)/2})}. f_1\otimes \ldots\otimes f_{(n-m+p)/2} \nonumber\\
&\otimes g_{(n+m-p+2)/2}\otimes \ldots\otimes g_n\nonumber\\
&:=f\stackrel{(n+m-p)/2}{\frown} g
\end{align*}
\bigbreak
Using this fact into the equation \ref{149}, we get by setting $r=(n+m-p)/2$ :
\begin{eqnarray}
I_n(f)I_m(g)&=&\sum_{p=\lvert n-m\lvert}^{n+m} I_p(f\stackrel{(n+m-p)/2}{\frown} g)\nonumber\\
&=&\sum_{r=0}^{n\wedge m} I_{n+m-2r}(f\stackrel{r}{\frown} g)
\end{eqnarray}
which concludes the proof.
\qed
\end{proof}
\begin{remark}
Another way of proving the product formula is to use the original approach of \"Ust\"unel in \cite{Us}, the previous Leibniz formula, the fact that
\newline
$\delta^p(f.1^{\otimes (p+1)})=I_p(f)$ for $f\in L^2(\mathbb{R}^p_+)$ and finally that the adjoint of the free Malliavin derivative $\nabla^p$ is $\delta^p$. It is however a less obvious and clear proof due to the heavy notations. We leave the details to the interested reader.
\newline
Indeed, without loss of generality, we can assume that $n<m$. 
By computing for $\phi \in \mathbb{D}^{\infty,\sigma}$, the following quantity:
\begin{equation}
\tau(I_n(f)I_m(g)\phi^*),
\end{equation}
and prove that it is equal to:
\begin{equation}
\sum_{r=0}^n \tau( I_{n+m-2r}(f\stackrel{r}{\frown} g) \phi^*),
\end{equation} gives the desired conclusion.
\end{remark}


\begin{thebibliography}{99}
\bibitem{BL}
\textsc{Christian Bender and Robert J.Elliott}, \textit{A Note on the Clark-Ocone Theorem for Fractional Brownian Motions with Hurst Parameter bigger than a Half}, Stochastics and Stochastics Reports, December 2003.
\bibitem{Biane}
{P. Biane}, \textit{Free hypercontractivity},
Commun. Math. Phys. 184, 457 – 474 (1997).
\bibitem{BS}
\textsc{P. Biane and R. Speicher (1998)}, \textit{Stochastic calculus with respect to free Brownian motion and analysis on Wigner space}, Prob. Theory Rel. Fields 112,
373–409.
\bibitem{BS2}
\textsc{P. Biane and R. Speicher (1998)}, \textit{Free diffusions, free entropy and free Fisher information
}, Annales de l'I.H.P. Probabilités et statistiques, Volume 37 (2001) no. 5, pp. 581-606.
\bibitem{bourg}
\textsc{Solesne Bourguin}, \textit{Vector-valued semicircular limits on the free Poisson chaos}, Electronic Communications in Probability, vol. 21, no. 55, 1-11, 2016. 
\bibitem{BC}
\textsc{Solesne Bourguin, Simon Campese}, \textit{Free quantitative fourth moment theorems on Wigner space}, International Mathematics Research Notices (IMRN), vol. 2018, no. 16, 4969-4990, 2018.
\bibitem{chl}
\textsc{Mireille Capitaine. Elton Hsu. Michel Ledoux},  \textit{Martingale Representation and a Simple Proof of Logarithmic Sobolev Inequalities on Path Spaces}, Electron. Commun. Probab. 2 71 - 81, 1997.
\bibitem{C}
\textsc{Guillaume Cébron}, \textit{A quantitative fourth moment theorem in free probability theory},
Advances in Mathematics, Volume 380, 2021, 107579.
\bibitem{CS}
\textsc{Ian Charlesworth and Dima Shlyakhtenko}, \textit{Free entropy dimension and regularity of
non-commutative polynomials}, Journal of Functional Analysis,
Volume 271, Issue 8, 15 October 2016, Pages 2274-2292.

\bibitem{CN}
\textsc{Ian Charlesworth, Brent Nelson},
\textit{Free Stein Irregularity and Dimension}, Journal of operator theory, Volume 85, Issue : 1, ISSN: 0379-4024.
\bibitem{CipS}
\textsc{Fabio Cipriani and Jean-Luc Sauvageot}, \textit{Derivations as square root of Dirichlet form}, Journal of Functional Analysis, Volume 201, Issue 1, 20 June 2003, Pages 78-120.

\bibitem{CSh}
\textsc{Alain Connes, Dimitri Shlyakhtenko}, \textit{$L^2$-Homology for von Neumann Algebras.} J. reine angew. Math.586(2005), 125—168.
\bibitem{DAB}
\textsc{Yoann Dabrowski}, \textit{A note about proving non-$\Gamma$ under a finite non-microstates free Fisher information assumption}, Journal of Functional Analysis,
Volume 258, Issue 11, 1 June 2010, Pages 3662-3674.
\bibitem{Dab10}
\textsc{Yoann Dabrowski}, \textit{A non-commutative Path Space approach to stationary free Stochastic Differential Equations}, arxiv:1006.4351v2.
\bibitem{Dab14}
\textsc{Yoann Dabrowski}, \textit{A free stochastic partial differential equation}, Annales de l'I.H.P. Probabilités et statistiques, Tome 50 (2014) no. 4, pp. 1404-1455.
\bibitem{YGS}
\textsc{Yoann Dabrowski, Alice Guionnet and Dima Shlyakhtenko},  \textit{Free transport for convex potentials.}New Zealand Journal of Mathematics. Volume 52 (2021), 259–359.
\bibitem{Deya}
\textsc{Aurélien Deya, Ivan Nourdin}, \textit{Convergence of Wigner integrals to the tetilla law},
ALEA, Lat. Am. J. Probab. Math. Stat. 9, 101–127 (2012).
\bibitem{q-gauss}
\textsc{Aurélien Deya, Salim Norredine, Ivan Nourdin}, \textit{Fourth Moment Theorem and q-Brownian Chaos}, Communications in Mathematical Physics volume 321, pages113–134 (2013).
\bibitem{inv}
\textsc{Aurélien Deya, Ivan Nourdin}, \textit{Invariance principles for homogeneous sums of free random variables}, Bernoulli 20(2): 586-603 (May 2014).
\bibitem{Di}
\textsc{Charles-Philippe Diez}, \textit{Free Malliavin-Stein-Dirichlet method: multidimensional semicircular approximations and chaos of a quantum Markov operator}, arxiv:2211.07595v1.
\bibitem{Ge}
\textsc{Liming Ge}, \textit{Applications of free entropy to von Neumann algebras, II.}
Annals of Mathematics
Second Series, Vol. 147, No. 1 (Jan., 1998), pp. 143-157 (15 pages).
\bibitem{GS}
\textsc{Alice Guionnet, Dimitri Shlyakhtenko}, \textit{Free diffusion and matrix models with strictly convex inetraction}, Geometric and Functional Analysis, 18(6):1875-1916.
\bibitem{Haag}
\textsc{Uffe Haagerup}, \textit{An example of a non nuclear $C^*$-algebra, which has the metric approximation property}, Invent Math 50, 279–293 (1978).
\bibitem{ioa}
\textsc{Adrian Ioana with an appendix joint with Stefaan Vaes}, \textit{Cartan subalgebras of amalgamated free product II1 factors}, Ann. Sci. Ec. Norm. Super. (4) 48 (2015), no. 1, 71-130.
\bibitem{JWS}
\textsc{David Jekel, Wuchen Li and Dimitri Shlyakhtenko},  \textit{Tracial smooth functions of non-commuting variables and the free Wasserstein manifold}, Dissertationes Mathematicae 580 (2022), 1-150.
\bibitem{KNPS}
\textsc{Todd Kemp, Ivan Nourdin, Giovanni Peccati, Roland Speicher}, \textit{ Wigner chaos and fourth moment theorems}, Ann. Probab. 2012.
\bibitem{ML}
\textsc{Michel Ledoux}, \textit{ Chaos of a Markov operator}, The Annals of Probability,
2012, Vol. 40, No. 6, 2439–2459
DOI: 10.1214/11-AOP685.
\bibitem{LNual}
\textsc{Jorge.A L\'eon and David Nualart}, \textit{Clark-Ocone Formula for Fractional BrownianMotion with Hurst Parameter Less Than 1/2},
Stochastic Analysis and Applications, 24: 427–449, 2006.
\bibitem{MS}
\textsc{Tobias Mai, Roland Speicher}, \textit{ A note on the free and cyclic differential calculus}, Journal of Operator Theory, Volume 85, Issue 1, Winter 2021  pp. 183-215
\bibitem{MSW}
\textsc{Tobias Mai, Roland Speicher, Moritz Weber},
\textit{Absence of algebraic relations and of zero divisors under the assumption of finite non-microstates free Fisher information.} arXiv:1407.5715 [math.OA] (2014)
\bibitem{Mai}
\textsc{Tobias Mai}, \textit{Regularity of distributions of Wigner integrals}, Arxiv preprint arXiv:1512.07593.
\bibitem{Malli}
\textsc{Paul Malliavin}, \textit{Stochastic Analysis}, Grundlehren der mathematischen Wissenschaften 313. Springer Berlin, Heidelberg, 1997.
\bibitem{meyer}
\textsc{P. A. Meyer}, \textit{Notes sur les processus d’Ornstein-Uhlenbeck.
S\'eminaire de Probabilit\'es}, XVI, p. 95-133. Lecture Motes in Math.
Vol. 920. Springer, 1982.
\bibitem{EN}
\textsc{Evangelos A. Nikitopoulos}, \textit{ Article
Itô's formula for noncommutative $C^2$ functions of free Itô processes}, Documenta mathematica, Journal der Deutschen Mathematiker-Vereinigung, 27:1447-1507.
\bibitem{nv}
\textsc{Ivan Nourdin, Frederi Viens}, \textit{Density Formula and Concentration Inequalities with Malliavin Calculus}, Electron. J. Probab. 14: 2287-2309 (2009). DOI: 10.1214/EJP.v14-707.
\bibitem{NP-book} 
    \textsc{I. Nourdin and G. Peccati (2012)},  \textit{ Normal Approximations
with Malliavin Calculus From Stein's Method to Universality}, Cambridge
University Press.
\bibitem{np-p}
\textsc{Ivan Nourdin and Giovanni Peccati}, \textit{Poisson approximations on the free Wigner chaos},
Ann. Probab. 41(4): 2709-2723 (July 2013). DOI: 10.1214/12-AOP815.
\bibitem{NuZ}
\textsc{Nualart, D and Zakai, M:}
\textit{Generalized Stochastic Integrals and the Malliavin Calculus}, Probability Theory and Related Fields volume 73, pages255–280 (1986).
\bibitem{nualZ}
\textsc{Nualart, D and Zakai, M:}, \textit{Generalized multiple stochastic integrals and the representation of wiener functionals}, Stochastics Volume 23, 1988-Issue 3.
\bibitem{Nual}
\textsc{Nualart, D},
\textit{ Malliavin Calculus and Related Topics}, Springer, Second Edition.

\bibitem{popa}
\textsc{Sorin Popa}, \textit{On a class of type II1 factors
with Betti numbers invariants}, Annals of Mathematics, 163 (2006), 809–899.
\bibitem{sain}
\textsc{Allan Sinclair,} \textit{Finite on Neumann algebras and MASAS}, Cambridge University Press, 2008.
\bibitem{shige}
\textsc{Ichiro Shigekawa,} \textit{Derivatives of Wiener functionals and absolute continuity of induced measures}, J. Math. Kyoto Univ. 20(2): 263-289 (1980), DOI: 10.1215/kjm/1250522278.
\bibitem{S04}
\textsc{Dima Shlyakhtenko,} \textit{ Lower estimates on microstates free entropy dimension}, Anal. PDE 2(2): 119-146 (2009). DOI: 10.2140/apde.2009.2.119.
\bibitem{S03}
\textsc{Dima Shlyakhtenko,} \textit{ Some estimates for non-microstates free entropy dimension with applications to q-semicircular families}, September 2003, International Mathematics Research Notices 2004(51) 1404-1455.
\bibitem{Stein}
\textsc{Charles Stein,} \textit{ A bound for the error in the normal approximation to the distribution of a sum of dependent random variables}, Berkeley Symposium on Mathematical Statistics and Probability, 1972: 583-602 (1972).
\bibitem{stroock}
\textsc{D. W. Stroock}, \textit{ Homogeneous chaos revisited}, Seminaire de Probabilités XXI, p. 1-8. Lecture Notes in Math.
Vol. 1247. Springer, 1987.
\bibitem{Us}
\textsc{Ali Süleyman Üstünel}, \textit{A sophisticated proof of the multiplication formula for multiple Wiener integrals},
arXiv:1411.4877.
\bibitem{ust}
\textsc{Ali Süleyman Üstünel}, \textit{Analysis on Wiener Space and Applications},
 arXiv:1003.1649v2.
 \bibitem{take}
 \textsc{Masamichi Takesaki}, {\em Conditional expectations in von Neumann algebras},
 Journal of Functional Analysis,
Volume 9, Issue 3, March 1972, Pages 306-321.
\bibitem{Voic1}
\textsc{Dan Voiculescu},\textit{ A note on cyclic gradients}, Indiana University Mathematics Journal,
Vol. 49, No. 3 (Fall, 2000), pp. 837-841.
\bibitem{Voic2}
\textsc{Dan Voiculescu}, \textit{ The analogues of entropy and of fisher's information measure in free probability theory III: The absence of Cartan subalgebras}, Geometric and Functional Analysis (GAFA), 6, pages 172–199 (1996).
\bibitem{V}
\textsc{Dan Voiculescu}, \textit{The analogues of entropy and of Fisher's information measure in free probability theory
V. Noncommutative Hilbert Transforms}, Invent. math. 132, 189±227 (1998). 



\end{thebibliography}
\end{document}